\documentclass[hidelinks,onefignum,onetabnum]{siamart220329}
\ifpdf
\hypersetup{
  pdftitle={Fully discrete energy-dissipative and conservative discrete gradient particle methods for a class of continuity equations},
  pdfauthor={J. Hu, S.Q. Van Fleet,  A. T. S. Wan}
}
\fi

\usepackage[page,header]{appendix}
\usepackage{titletoc}

\newcommand{\rd}{\,\mathrm{d}}

\newcommand\ip[2]{\left\langle #1,#2 \right\rangle}
\newcommand\ipB[2]{\mathcal{B}[f]\left( #1,#2 \right)}

\newcommand{\bX}{\mathbf X}
\newcommand{\bbx}{\bm x}
\newcommand{\bby}{\bm y}

\numberwithin{equation}{section}
\usepackage{amsfonts}
\usepackage{amsmath}

\usepackage{fancyhdr}
\usepackage{titlesec}
\usepackage{indentfirst}
\usepackage{booktabs}
\usepackage{verbatim}
\usepackage{color}
\usepackage{bm}
\usepackage{xcolor}

\usepackage{listings}

\usepackage{caption}

\usepackage{amssymb}
\usepackage{graphicx}
\usepackage{epstopdf}
\usepackage{algorithmic}
\ifpdf
  \DeclareGraphicsExtensions{.eps,.pdf,.png,.jpg}
\else
  \DeclareGraphicsExtensions{.eps}
\fi

\newcommand{\TheTitle}{Fully discrete energy-dissipative and conservative discrete gradient particle methods for a class of continuity equations}

\newcommand{\TheAuthors}{J. Hu, S.Q. Van Fleet, A.T.S. Wan}

\headers{Dissipative and conservative discrete gradient particle methods}{\TheAuthors}

\title{{\TheTitle}\thanks{%Submitted to the editors on {\today}.
All authors contributed equally. \funding{JH and SQVF 's research was supported in part by AFOSR grant FA9550-21-1-0358, NSF grant DMS-2409858, and DOE grant DE-SC0023164. SVF was also supported by the Pacific Institute for the Mathematical Sciences (PIMS). The research and findings may not reflect those of the Institute.}}}

\author{
      Jingwei Hu\thanks{Department of
  Applied Mathematics, University of Washington, Seattle, WA 98195, USA\\
    (\email{hujw@uw.edu}).}
    \and
    Samuel Q. Van Fleet\thanks{Department of
  Applied Mathematics, University of Washington, Seattle, WA 98195, USA\\
    (\email{svfleet@uw.edu}).}
   \and
    Andy T. S. Wan\thanks{Department of Applied Mathematics, University of California, Merced, Merced CA 95343, USA (\email{andywan@ucmerced.edu).}}
}

\usepackage{amsopn}

\begin{document}

\maketitle

% REQUIRED
\begin{abstract}
Structure-preserving particle methods have recently been proposed for a class of nonlinear continuity equations, including aggregation-diffusion equation in [J. Carrillo, K. Craig, F. Patacchini, \emph{Calc. Var.}, 58 (2019), pp. 53] and the Landau equation in [J. Carrillo, J. Hu., L. Wang, J. Wu, \emph{J. Comput. Phys. X}, 7 (2020), pp. 100066]. One common feature to these equations is that they both admit some variational formulation, which upon proper regularization, leads to particle approximations dissipating the energy and conserving some quantities simultaneously at the semi-discrete level. In this paper, we formulate continuity equations with a density dependent bilinear form associated with the variational derivative of the energy functional and prove that appropriate particle methods satisfy a compatibility condition with its regularized energy. This enables us to utilize discrete gradient time integrators and show that the energy can be dissipated and the mass conserved simultaneously at the fully discrete level. In the case of the Landau equation, we prove that our approach also conserves the momentum and kinetic energy at the fully discrete level. Several numerical examples are presented to demonstrate the dissipative and conservative properties of our proposed method.
\end{abstract}

% REQUIRED

% REQUIRED

\section{Introduction}
\label{sec:Intro}
We are interested in structure-preserving discretizations to a class of nonlinear continuity equations describing the evolution of an unknown density function $f(t,\bbx)$:
\begin{equation}
    \partial_t f(t,\bm x) + \nabla_{\bbx} \cdot\left(f(t,\bm x) \bm U[f](t,\bm x)\right)=0, \quad t>0, \ \bm x\in \mathbb{R}^d,\label{eq: NLContEqn}
\end{equation} where $\bm U[f]$ is a velocity field depending on $f$. Specifically, we will be employing a particle method in space and discrete gradient integrator in time to preserve simultaneously dissipative and conservative quantities for a class of \eqref{eq: NLContEqn}. 

As a first example, the aggregation-diffusion equation has the form \eqref{eq: NLContEqn} with 
\begin{equation}
\bm U[f](t,\bm x) = -\nabla_{\bbx} \frac{\delta E_A}{\delta f}(t,\bm x),
\end{equation}
where $\frac{\delta E_A}{\delta f}(t,\bm x)$ is the variational derivative of the functional
\begin{equation}\label{eq:energy functional}
E_A[f](t)=\int_{\mathbb{R}^d} H(f)+V f+\frac{1}{2}(W*f) f\rd\bm x.
\end{equation}Or more explicitly, 
\begin{equation} \label{eq:GFP}
\partial_tf(t,\bm x)= \nabla_{\bbx} \cdot [f(t,\bm x)\nabla_{\bbx}(H'(f)(t,\bm x)+V(\bm x)+(W*f)(t,\bm x))],
\end{equation} where $H(f)$ is the internal energy, $V(\bbx)$ is an external potential, and $W(\bbx)$ is a symmetric interaction potential. \eqref{eq:GFP} includes a large class of models such as the heat equation ($H(f)=f\log f$, $V=W=0$); porous medium equation ($H(f)=\frac{1}{m-1}f^m$, $m>1$, $V=W=0$); linear Fokker-Planck equation ($H(f)=f\log f$, $W=0$, and $V$ is some given function); Keller-Segel equation  ($H(f)=f\log f$, $V=0$, and $W$ is the Newtonian potential), and many others (see  \cite{CCY19} for more examples).

The second example we  consider is the Landau equation, which is a kinetic equation describing charged particle collisions \cite{Landau37, Villani02}. It is given by \eqref{eq: NLContEqn} with
\begin{equation}
    \bm U[f](t,\bm x) = -\int_{\mathbb{R}^d} A(\bm x-\bm y)\left(\nabla_{\bm x}\frac{\delta E_L}{\delta f}(t,\bm x) - \nabla_{\bm y} \frac{\delta E_L}{\delta f}(t,\bm y)\right)f(t,\bm y) \rd\bm y, \label{eq:LandauVelocity}
\end{equation}
where $\frac{\delta E_L}{\delta f}(t,\bm x)$ is the variational derivative of the functional
\begin{equation} \label{eq:energy functional1}
E_L[f](t) = \int_{\mathbb{R}^d} f \log f \rd\bm x,
\end{equation} and $A$ is a $d\times d$ symmetric positive semi-definite  matrix with its entries given by \\$a_{ij}(\bbx)=C|\bbx|^{\gamma}(|\bbx|^2\delta_{ij}-x_ix_j)$,
for $-d\leq \gamma \leq 1$. More explicitly, Landau equation is
\begin{equation}\label{eq:Landau1}
\partial_t f(t,\bm x)= \nabla_{\bbx} \cdot  \left[ \int_{\mathbb{R}^d} A(\bm x-\bm y)\left(\nabla_{\bm x} \log f(t,\bm x)-\nabla_{\bm y} \log f(t,\bm y)\right)f(t,\bm x) f(t,\bm y)\rd\bm y\right].
\end{equation} 

It is well-known that \eqref{eq:GFP} and \eqref{eq:Landau1} possess both dissipative and conservative  quantities. To see this directly and in anticipation with our structure-preserving discretization, we propose an alternative view of using a symmetric, negative semi-definite bilinear form. Specifically, if the velocity field $\bm U[f]$ of \eqref{eq: NLContEqn} can be identified with a bilinear form associated with the variational derivative of some functional, then \eqref{eq: NLContEqn} has a dissipative quantity and may possess some conserved quantities.
\begin{lemma}
    Given a solution $f(t,\bbx)$ of \eqref{eq: NLContEqn}, suppose that there exists a symmetric, negative semi-definite\footnote{Here, $B[f](\cdot,\cdot)$ is symmetric if $\ipB{\phi}{\psi} = \ipB{\psi}{\phi}$ and negative semi-definite if $\ipB{\phi}{\phi}\leq 0$ for all $\phi,\psi$.} bilinear form $B[f](\cdot,\cdot)$ satisfying
    \begin{equation}
    \ipB{\phi}{\frac{\delta E}{\delta f}}= \ip{\nabla_{\bbx} \phi}{f \bm U[f]} , ~~\phi \in C_c^\infty(\mathbb{R}^d), \label{eq: bilinearForm}
    \end{equation}
    where  $E[f](t)=\int_{\mathbb{R}^d} F(t,\bm x, f, g*f)\rd\bm x$ is a functional\footnote{In principle, one can also consider integrands depending on derivatives of $f$. For clarity and specificity to our examples, we have chosen here to focus on such cases only.} for some sufficiently smooth function $g$, and $\ip{\cdot}{\cdot}$ denotes the usual inner product in $L^2(\mathbb{R}^d)$. Then $\frac{\rd}{\rd t}E[f] \leq 0$.
    Moreover, if there exists some functional $I[f](t)=\int_{\mathbb{R}^d} G(t,\bm x, f, g*f) \rd\bm x$ such that $\ipB{\frac{\delta I}{\delta f}}{\frac{\delta E}{\delta f}}=0,$ then $\frac{\rd}{\rd t}I[f](t)=0$. \label{lem: bilinearForm}
\end{lemma}
\begin{proof}From \eqref{eq: NLContEqn} and negative semi-definiteness of $B[f]$,
    \begin{align*}
        \frac{\rd}{\rd t}E[f] &= \ip{\frac{\delta E}{\delta f}}{\partial_t f} =  \ip{\frac{\delta E}{\delta f}}{ -\nabla_{\bbx} \cdot\left(f \bm U[f]\right)}= \ip{\nabla_{\bbx} \frac{\delta E}{\delta f}}{f\bm U[f]} \\
        &= \ipB{\frac{\delta E}{\delta f}}{\frac{\delta E}{\delta f}} \leq 0.
    \end{align*} 
    Similarly, by hypothesis of $I$,
    \begin{align*}
        \frac{\rd}{\rd t}I[f] &= \ip{\frac{\delta I}{\delta f}}{\partial_t f} =
        \ip{\frac{\delta I}{\delta f}}{ -\nabla_{\bbx} \cdot\left(f \bm U[f]\right)}
        = \ip{\nabla_{\bbx} \frac{\delta I}{\delta f}}{f\bm U[f]}\\
        &=
        \ipB{\frac{\delta I}{\delta f}}{\frac{\delta E}{\delta f}} =0.
    \end{align*}
\end{proof}
We note that \eqref{eq: bilinearForm} is closely related to Hamiltonian PDEs where it would instead be a skew-symmetric bilinear form, satisfying the so-called Jacobi identity \cite{olver93}. We now apply \Cref{lem: bilinearForm} to show that \eqref{eq:GFP} and \eqref{eq:Landau1} possess both dissipative and conservative quantities.

For the aggregation-diffusion equation \eqref{eq:GFP}, the associated bilinear form is
\begin{equation}
\ipB{\phi}{\psi}=-\int_{\mathbb{R}^d}\nabla_{\bbx}\phi(\bbx) \cdot \nabla_{\bbx}\psi(\bbx)f(t,\bbx)\rd{\bbx},
\end{equation}
which can be readily verified to be symmetric and negative semi-definite. Taking $I_A[f]=\int_{\mathbb{R}^d}f\rd{\bbx}$ and $E_A[f]$ as in \eqref{eq:energy functional}, so $\frac{\delta I_A}{\delta f}=1$, then
$B[f](\frac{\delta I_A}{\delta f},\frac{\delta E_A}{\delta f})=0$. By \Cref{lem: bilinearForm}, one thus has the energy dissipation and mass conservation:
$$
\frac{\rd E_A}{\rd t}\leq0, \quad \frac{\rd I_A}{\rd t}=0.
$$

For the Landau equation \eqref{eq:Landau1}, the associated bilinear form is
\begin{equation}
   \ipB{\phi}{\psi} =- \int_{\mathbb{R}^d\times\mathbb{R}^d} (\nabla_{\bbx} \phi (\bbx))^T A(\bm x-\bm y)\left(\nabla_{\bm x}\psi(\bbx) -\nabla_{\bm y}\psi(\bby)\right) f(t,\bm x) f(t,\bm y) \rd\bm x \rd\bm y.
\end{equation} 
Using $A(\bbx)=A(-\bbx)$, one can further show that
\begin{align*}
   \ipB{\phi}{\psi} =&- \frac{1}{2}\int_{\mathbb{R}^d\times\mathbb{R}^d} (\nabla_{\bbx} \phi (\bbx)-\nabla_{\bby} \phi (\bby))^T A(\bm x-\bm y) \\
   &\qquad\times \left(\nabla_{\bm x}\psi(\bbx) -\nabla_{\bm y}\psi(\bby)\right) f(t,\bm x) f(t,\bm y) \rd\bm x \rd\bm y,
\end{align*} 
from which it follows immediately that $B[f]$ is symmetric and negative semi-definite (since $A=A^T$ and is positive semi-definite). Taking $\bm I_L[f]=\int_{\mathbb{R}^d}  (1,\bbx,|\bbx|^2/2)^Tf\rd{\bbx}$ and $E_L[f]$ as in \eqref{eq:energy functional1}, so $\frac{\delta \bm I_L}{\delta f}=(1,\bbx,|\bbx|^2/2)^T$, then $B[f](\frac{\delta \bm I_L}{\delta f},\frac{\delta E_L}{\delta f})=\bm 0$ (the last one is due to $A(\bbx)\bbx=\bm 0$). Thus by \Cref{lem: bilinearForm}, one obtains the energy dissipation\footnote{In the kinetic literature, this is often called the entropy dissipation since \eqref{eq:energy functional1} is the Boltzmann entropy \cite{Villani02}.} and simultaneously the conservation of mass, momentum, and kinetic energy:
\begin{equation*}
\frac{\rd E_L}{\rd t}\leq0, \quad \frac{\rd \bm I_L}{\rd t}=0.
\end{equation*}

To preserve the above dissipative and conservative properties of \eqref{eq:GFP} and \eqref{eq:Landau1}, we propose using the regularized particle approach of \cite{CCP19} and \cite{CHWW20}. As we will discuss in \Cref{sec:Particle method}, this will lead to a system of ODEs which are discretized version of the bilinear form \eqref{eq: bilinearForm}. Thanks to a \emph{compatibility condition} satisfied by such particle method and its regularized energy functional, these ODEs inherit the dissipative and conservative properties of the continuity equations. From this, we can then utilize \emph{discrete gradient integrators}\footnote{A discrete gradient approximation $\overline{\nabla E}:\mathbb{R}^k\rightarrow \mathbb{R}^k$ of $\nabla E$ is a vector field satisfying,
\begin{align}
    \overline{\nabla E}(\bm Y, \bm X)^T(\bm Y-\bm X) &= E(\bm Y)-E(\bm X), ~~\bm X, \bm Y\in\mathbb{R}^k. \label{eq: discGradProp}%\\
    %\overline{\nabla E}(\bm x, \bm x) &= \nabla E(\bm x).
\end{align} Examples include the Itoh-Abe discrete gradient \cite{IA88}, midpoint discrete gradient \cite{G96} and mean-value \cite{HLV83} or average vector field discrete gradient \cite{QM08}.} and prove that the resulting fully discrete schemes for \eqref{eq:GFP} and \eqref{eq:Landau1} satisfy respectively 
\begin{align*}
    E_A(\bm X^{n+1}) \leq E_A(\bm X^{n}),& \quad I_A(\bm X^{n+1}) = I_A(\bm X^{n}),\\
    E_L(\bm X^{n+1}) \leq E_L(\bm X^{n}),& \quad \bm I_L(\bm X^{n+1}) = \bm I_L(\bm X^{n}),
\end{align*} 
where $\bm X^n$ is the particle solution at time step $t_n$ to be detailed in \Cref{sec:ODEform}.

We note that structure-preserving discretizations have been proposed for constant skew-symmetric or negative semi-definite operators on Hilbert spaces in \cite{CelleAE12}, where discrete gradient integrators were applied to a wide variety of PDEs with either conservative or dissipative quantities independently. In contrast, we consider continuity equations possessing both dissipative and conservative quantities \emph{simultaneously}. In the geometric numerical integration literature, this scenario is rather exceptional and has so far not been well studied using discrete gradient integrators\footnote{In general, even preserving $m$ conserved quantities simultaneously entails working with $(m+1)$-th order skew-symmetric tensor with discrete gradient integrators \cite{MQR99}.} to the best knowledge of the authors. Moreover, we consider the associated symmetric bilinear form which is \emph{nonconstant} (i.e., it can depend on the solution) and it can be \emph{degenerate}\footnote{A degenerate bilinear form means there exists a non-zero $u$ such that $\ipB{u}{v}=0$ for all $v$, which is in contrast to inner product being nondegenerate. For instance, the bilinear form associated with the Landau equation is degenerate since $\ipB{|\bm x|^2}{\psi}=0$ for all $\psi$ due to $A(\bm x)\bm x=0$.} and does not necessarily need to be an inner product. Nevertheless, with appropriate particle discretization, we will show that the discrete gradient integrators do in fact preserve both dissipative and conservative quantities for the aggregation-diffusion equation \eqref{eq:GFP} and Landau equation \eqref{eq:Landau1}.

On the other hand in the plasma physics literature, we note that similar particle discretizations based on discrete gradient integrators have also recently been proposed for the Landau equation. Specifically, \cite{Hirvijoki21} extended a particle scheme of \cite{CHWW20} using discrete gradient integrator for the Landau equation. While our proposed particle scheme with mean-value discrete gradient integrator is similar to that of \cite{Hirvijoki21}, our approach based on a symmetric negative-semidefinite bilinear form enables us to work with a broader class of PDEs with dissipative and conservative quantities. Moreover, we provide more extensive numerical studies on the Landau equation using discrete gradient integrators, while \cite{Hirvijoki21} only showed numerical results involving a midpoint velocity integrator. Also, \cite{JKHP24} proposed a structure-preserving particle approach for the Landau equation by exploiting the so-called metriplectic structure, with spline and finite element approximation in space and midpoint discrete gradient integrator in time. In contrast, we employ Gaussian basis approximations which avoid the need to invert a mass matrix appearing in the formulation of \cite{JKHP24}. Moreover, we showcase extensive numerical studies (absent in \cite{JKHP24}) and utilize quadrature approximation of the mean-value discrete gradient integrator to ensure robustness and avoid large round-off errors with small divisors that can occur with the midpoint discrete gradient.

This paper is organized as follows. In \Cref{sec:Particle method}, we review specific particle discretizations for the aggregation-diffusion equation \eqref{eq:GFP} and Landau equation \eqref{eq:Landau1}, as well as proving an important compatibility condition which enables the use of discrete gradient integrators. We then show in \Cref{sec:ODEform} how discrete gradient integrators can preserve for a class of ODEs with 
discretized dissipative and conservative quantities arising from the particle methods. Finally in \Cref{sec:Examples}, we present various numerical results to demonstrate that our fully discrete scheme preserves the relevant dissipative and conservative quantities simultaneously.

%%%%%%%%%%%%%%%%%%%%%%%%%%%%%%%%%%%%%%%
\section{Particle method}
\label{sec:Particle method}

In the classical deterministic particle method \cite{Chertock}, one seeks a weak solution of \eqref{eq: NLContEqn} as follows:
\begin{equation} \label{fparticle}
f(t,\bbx) \approx f^N(t,\bbx) := \sum_{p=1}^N w_p \delta (\bbx -\bbx_p(t)),
\end{equation}
where $\{w_p,\bbx_p(t)\}_{p=1}^N$ are the particle weights (which are positive and constant) and positions satisfying
\begin{equation} \label{eq:particleMethod}
\frac{\rd \bbx_p}{\rd{t}}=\bm U[f^N](t,\bbx_p), \quad \bbx_p(0)=\bbx_i,\quad w_p=h^d f^0(\bbx_i).
\end{equation}
Here we choose a large enough computational domain $[-L,L]^d$ (such that the solution is negligible close to the boundary) and partition the domain into $N=M^d$ uniform cells with cell size $h=2L/M$, and choose initial particle positions as the cell center $\{\bbx_i\}$ and particle weights as the cell mass approximated by the midpoint quadrature. 

In view of the bilinear form in \eqref{eq: bilinearForm}, we choose the velocity field $\bm U[f^N]$ to hold \eqref{eq: bilinearForm} at the particle level, by regularizing the functional $E[f]$ so that
\begin{equation} \label{eq: bilinearReg} \mathcal{B}[f^N]\left(\phi,\frac{\delta E^{\varepsilon}}{\delta f}[f^N]\right)= \ip{\nabla_{\bbx} \phi}{f^N \bm U[f^N]} , ~~\phi \in C_c^\infty(\mathbb{R}^d). 
\end{equation}
For the aggregation-diffusion equation \eqref{eq:GFP}, this leads to
\begin{equation} \label{schemeA}
\bm U[f^N](t,\bbx_p)\approx-\nabla_{\bbx} \frac{\delta E_A^{\varepsilon}}{\delta f}[f^N](\bbx_p),
\end{equation}
and for the Landau equation \eqref{eq:Landau1}, this leads to
\begin{equation} \label{schemeL}
\bm U[f^N](t,\bbx_p)\approx -\sum_{q=1}^N w_q A(\bbx_p-\bbx_q) \left(\nabla_{\bbx} \frac{\delta E_L^{\varepsilon}}{\delta f}[f^N](\bbx_p)-\nabla_{\bbx} \frac{\delta E_L^{\varepsilon}}{\delta f}[f^N](\bbx_q)\right).
\end{equation}
Following \cite{CCP19} and \cite{CHWW20}, the regularized energy $E_A^{\varepsilon}$ and $E_L^{\varepsilon}$  can be chosen as\footnote{There are alternative ways for regularization; such as choosing $H^{\varepsilon}(f)= f\log (f*\varphi_{\varepsilon})$ and $H^{\varepsilon}(f)= \frac{1}{m-1}f(f*\varphi_{\varepsilon})^{m-1}$, respectively. If the potential $W$ is singular, one might also need to regularize $W$ such that it becomes $W*\varphi_{\varepsilon}$ \cite{CB15}. Our proposed method also works for these regularizations. For concreteness, we use the regularization presented in the main text for numerical demonstration.}
\begin{equation} \label{Hreg}
E_A^{\varepsilon}[f]=\int_{\mathbb{R}^d }H^{\varepsilon}(f)+Vf+\frac{1}{2}(W*f)f\rd{\bbx}, \quad E_L^{\varepsilon}[f]=\int_{\mathbb{R}^d } H^{\varepsilon}(f)\rd{\bbx},
\end{equation}
where in the case of $H(f)=f \log f$, $H^{\varepsilon}(f):= (f*\varphi_{\varepsilon})\log (f*\varphi_{\varepsilon})$; and in the case of $H=\frac{1}{m-1}f^m$, $H^{\varepsilon}(f):= \frac{1}{m-1}(f*\varphi_{\varepsilon})^m$. Here $\varphi_{\varepsilon}(\bbx)$ is the mollifier function with a parameter $\varepsilon$, often chosen as the Gaussian function $\varphi_{\varepsilon}(\bbx)=(2\pi \varepsilon)^{-d/2}\exp\left(-\frac{|\bbx|^2}{2\varepsilon}\right)$. Computing $\nabla_{\bbx}\frac{\delta E_A^{\varepsilon}}{\delta f}$ and $\nabla_{\bbx}\frac{\delta E_L^{\varepsilon}}{\delta f}$, and evaluating them at $f^N$ yields the explicit expressions used in \eqref{schemeA} and \eqref{schemeL}. Specifically, 
\begin{align} 
\nabla_{\bbx} \frac{\delta E_A^{\varepsilon}}{\delta f}[f^N](\bbx_p)&=h^{\varepsilon}(\bbx_p)+\nabla_{\bbx} V(\bbx_p)+\sum_{q=1}^N w_q\nabla_{\bbx} W(\bbx_p-\bbx_q), \label{schemeA1}\\
\nabla_{\bbx} \frac{\delta E_L^{\varepsilon}}{\delta f}[f^N](\bbx_p)&=h^{\varepsilon}(\bbx_p),\label{schemeL1}
\end{align}
where for $H^{\varepsilon}(f)= (f*\varphi_{\varepsilon})\log (f*\varphi_{\varepsilon})$, 
\begin{align}
h^{\varepsilon}(\bbx_p):=\int_{\mathbb{R}^d} \nabla_{\bbx} \varphi_{\varepsilon}(\bbx_p-\bbx)\log \left(\sum_q w_q \varphi_{\varepsilon}(\bbx-\bbx_q)\right)\rd{\bbx},\label{eq:heps1}
\end{align}
and for $H^{\varepsilon}(f)= \frac{1}{m-1}(f*\varphi_{\varepsilon})^m$,
\begin{equation}
h^{\varepsilon}(\bbx_p):=\frac{m}{m-1}\int_{\mathbb{R}^d} \nabla_{\bbx} \varphi_{\varepsilon}(\bbx_p-\bbx)\left(\sum_q w_q \varphi_{\varepsilon}(\bbx-\bbx_q)\right)^{m-1}\rd{\bbx}.\label{eq:heps2}
\end{equation}
\subsection{Compatibility condition of particle method for a class of energy functional}
Motivated by the aforementioned cases, we introduce a general energy functional of the form
\begin{equation} \label{eq:regularE}
E[f](t) = \int_{\mathbb{R}^d} F(t,\bm x, (f*u)(t,\bm x)) + f(t,\bm x)G(t,\bm x,(f*v)(t,\bm x)) \rd \bm x,
\end{equation}
where $u,v$ are some fixed symmetric\footnote{Here, we say that $g$ is symmetric if $g(-\bm x)=g(\bm x)$.} functions. We show in the following that energy functional $E[f]$ of such form satisfies a special \emph{compatibility} condition, which will enable us to utilize discrete gradient integrators to preserve the relevant dissipative and conservative quantities. Specifically, we have
\begin{lemma}
    Let $u,v$ be fixed symmetric integrable functions on $\mathbb{R}^d$ with integrable $\nabla u, \nabla v$. Also, let $F(t,\bm x, y)$, $G(t,\bm x, y)$, $F_y(t,\bm x, y)$, $G_y(t,\bm x, y)$ be integrable functions on $(0,\infty)\times\mathbb{R}^d\times \mathbb{R}$. The energy functional defined in \eqref{eq:regularE} satisfies 
    \begin{equation}
        \frac{1}{w_p}\nabla_{\bm x_p} E[f^N](\bm x_1,\dots,\bm x_N) = \nabla_{\bm x}\frac{\delta E}{\delta f}[f^N](\bm x_p),
    \end{equation} \label{lem:compCond}
 where $f^N$ is the particle approximation given in \eqref{fparticle}.
\end{lemma}
\begin{proof}
For brevity and clarity, we suppress writing dependence on $t$ in $\bm x_q=\bm x_q(t)$. As $(f^N*u)(t,\bm x)=\displaystyle\sum_{q=1}^N w_q u(\bm x-\bm x_q)$ and similarly for $f^N*v$, evaluating $E[f^N]$ yields
    \begin{align*}
        E[f^N] &= \int_{\mathbb{R}^d} F\left(t,\bm x, \sum_{q=1}^N w_q u(\bm x-\bm x_q)\right) \rd\bm x+\sum_{q=1}^N w_q G\left(t,\bm x_q, \sum_{r=1}^N w_r v(\bm x_q-\bm x_r)\right).
    \end{align*} Taking the gradient with respect to $\bm x_p$ gives
        \begin{align}
        \label{eq:LHS-compCond}&\nabla_{\bm x_p} E[f^N](\bm x_1,\dots,\bm x_N)\\
        &=\int_{\mathbb{R}^d} \nabla_{\bm x_p} F\left(t,\bm x, \sum_{q=1}^N w_q u(\bm x-\bm x_q)\right)\rd\bm x \nonumber\\
        &\hskip 2mm+ \nabla_{\bm x_p}\left[w_p G\left(t,\bm x_p,\sum_{r=1}^N w_r v(\bm x_p-\bm x_r)\right)+\sum_{\substack{q=1\\ q\neq p}}^N w_q G\left(t,\bm x_q,\sum_{r=1}^N w_r v(\bm x_q-\bm x_r)\right)\right] \nonumber\\
        &=w_p\int_{\mathbb{R}^d} F_y\left(t,\bm x, \sum_{q=1}^N w_q u(\bm x-\bm x_q)\right) \nabla_{\bbx} u(\bm x_p-\bm x)\rd\bm x\nonumber\\
        &\hskip 2mm+ w_p\left[\nabla_{\bbx}G\left(t,\bm x_p,\sum_{r=1}^N w_r v(\bm x_p-\bm x_r)\right) +G_y\left(t,\bm x_p,\sum_{r=1}^N w_r v(\bm x_p-\bm x_r)\right) \right.\nonumber\\
        &\hskip 6mm \left.\times\sum_{\substack{r=1\\r\neq p}}^N w_r \nabla_{\bbx} v(\bm x_p-\bm x_r) +\sum_{\substack{q=1\\q\neq p}}^N w_qG_y\left(t,\bm x_q,\sum_{r=1}^N w_r v(\bm x_q-\bm x_r)\right) \nabla_{\bbx} v(\bm x_p-\bm x_q)\right].
        \nonumber
    \end{align}
    On the other hand, evaluating the variational derivative of $E[f]$ gives  \begin{align*}
       &\left.\frac{\rd}{\rd\alpha} E[f+\alpha \eta]\right|_{\alpha =0}\\
        &= \int_{\mathbb{R}^d} F_y\left(t, \bm x, (f*u)(t,\bm x)\right)(\eta*u)(\bm x) \rd\bm x\\
        &\hskip 0mm+\int_{\mathbb{R}^d} f(t,\bm x)G_y\left(t, \bm x, (f*v)(t,\bm x)\right)(\eta*v)(\bm x) \rd\bm x + \int_{\mathbb{R}^d} \eta(\bm x)G\left(t,\bm x,(f*v)(t,\bm x)\right) \rd\bm x\\
        &= \int_{\mathbb{R}^d} \eta(\bm y)\dfrac{\delta E}{\delta f}(t,\bm y) \rd\bm y,
    \end{align*}
    where $\displaystyle \frac{\delta E}{\delta f}=F_y(t,\cdot,f*u)*u+\left(fG_y(t,\cdot,f*v)\right)*v+G(t,\cdot,f*v)$.
    Computing the gradient with respect to $\bm x$ of the variational derivative of $E$ yields\begin{align*}
        \nabla_{\bm x}\frac{\delta E}{\delta f}[f](\bm x) &= \int_{\mathbb{R}^d} F_y\left(t, \bm y, (f*u)(t,\bm y)\right)\nabla_{\bbx}u(\bm x-\bm y) \rd\bm y \\
        &\hskip 4mm+ \int_{\mathbb{R}^d} f(t,\bm y)G_y(t,\bm y,(f*v)(t,\bm y))\nabla_{\bbx} v(\bm x-\bm y)\rd\bm y\\
        &\hskip 4mm+ \nabla_{\bbx} G(t,\bm x,(f*v)(\bm x))+G_y(t,\bm x,(f*v)(\bm x))(f*\nabla_{\bbx} v)(\bm x).
    \end{align*}Thus, evaluating $\nabla_{\bm x}\dfrac{\delta E}{\delta f}[f^N]$ at $\bm x=\bm x_p$ gives
    \begin{align}
        \label{eq:RHS-compCond}
        \nabla_{\bm x}\frac{\delta E}{\delta f}[f^N](\bm x_p) &= \int_{\mathbb{R}^d} F_y\left(t, \bm y, \sum_{q=1}^N w_q u(\bm y-\bm x_q)\right)\nabla_{\bbx}u(\bm x_p-\bm y) \rd\bm y\\
        &\hskip 4mm+ \sum_{q=1}^N w_q G_y\left(t,\bm x_q,\sum_{r=1}^N w_r v(\bm x_q-\bm x_r)\right)\nabla_{\bbx} v(\bm x_p-\bm x_q) \nonumber\\
        &\hskip 4mm+\nabla_{\bbx}G\left(t,\bm x_p,\sum_{q=1}^N w_q v(\bm x_p-\bm x_q)\right)\nonumber\\
        &\hskip 4mm+ G_y\left(t,\bm x_p,\sum_{q=1}^N w_q v(\bm x_p-\bm x_q)\right)\left(\sum_{r=1}^N w_r \nabla_{\bbx} v(\bm x_p-\bm x_r)\right).\nonumber
    \end{align}
 Comparing both sides of \eqref{eq:LHS-compCond} and \eqref{eq:RHS-compCond} and using  $\nabla_{\bbx} v(\bm 0) = \bm 0$ (by symmetry of $v$) shows the desired result.
\end{proof}

We note \Cref{lem:compCond} covers the regularized energies\footnote{\Cref{lem:compCond} can be applied to other regularized energy mentioned earlier. For example, if instead $E_L^\varepsilon[f]=\int_{\mathbb{R}^d} f\log (f*\varphi_\varepsilon) \rd\bm x$ is used, then the associated integrand would be $G(t,\bm x,u)=\log u$.} defined in \eqref{Hreg}. Specifically, one can express the energy functional of the aggregation-diffusion equation as
\begin{align}
    E_A^\varepsilon[f]=\int_{\mathbb{R}^d} F(t,\bm x,f*\varphi_\varepsilon)+f(t,\bm x)G(t,\bm x,f*W)\rd\bm x, \label{eq:EA alt}
\end{align} with $F(t,\bm x,y)=y\log y$ or $\frac{1}{m-1}y^m$, $G(t,\bm x,y)=V(\bm x) + \frac{1}{2}y$. Similarly, the energy functional of the Landau equation can be written as
\begin{align}
    &E_L^\varepsilon[f]=\int_{\mathbb{R}^d} F(t,\bm x,f*\varphi_\varepsilon)\rd\bm x, \label{eq:EL alt}
\end{align} with $F(t,\bm x,y)=y\log y$.
Thus, applying \Cref{lem:compCond} to \eqref{eq:EA alt} and \eqref{eq:EL alt} gives the following identities.
\begin{corollary}  \label{cor}
Define the discretized energies
\begin{equation}
E_A^{\varepsilon}(\bbx_1,\dots,\bbx_N)=E_A^{\varepsilon}[f^N], \quad E_L^{\varepsilon}(\bbx_1,\dots,\bbx_N)=E_L^{\varepsilon}[f^N],
\end{equation}
where $E_A^{\varepsilon}[f]$ and $E_L^{\varepsilon}[f]$ are defined in (\ref{Hreg}). Then one has
\begin{align} 
&\frac{1}{w_p}\nabla_{\bbx_p}E_A^{\varepsilon}(\bbx_1,\dots,\bbx_N)=\nabla_{\bbx} \frac{\delta E_A^{\varepsilon}}{\delta f}[f^N](\bbx_p), \\
&\frac{1}{w_p}\nabla_{\bbx_p}E_L^{\varepsilon}(\bbx_1,\dots,\bbx_N)=\nabla_{\bbx} \frac{\delta E_L^{\varepsilon}}{\delta f}[f^N](\bbx_p),
\end{align}
where the right hand sides are given by \eqref{schemeA1} and \eqref{schemeL1}.
\end{corollary}

By \Cref{cor}, the particle method \eqref{eq:particleMethod} with velocity \eqref{schemeA} for the aggregation-diffusion equation \eqref{eq:GFP} can be written as
\begin{equation} \label{eq:schemeA2}
\frac{\rd \bbx_p}{\rd{t}}=-\frac{1}{w_p}\nabla_{\bbx_p}E_A^{\varepsilon}(\bbx_1,\dots,\bbx_N),
\end{equation}
and the particle method \eqref{eq:particleMethod} with velocity \eqref{schemeL} for the Landau equation \eqref{eq:Landau1} can be written as
\begin{equation} \label{eq:schemeL2}
\frac{\rd \bbx_p}{\rd{t}}=-\sum_{q=1}^N w_q A(\bbx_p-\bbx_q) \left(\frac{\nabla_{\bbx_p}E_L^{\varepsilon}(\bbx_1,\dots,\bbx_N)}{w_p}- \frac{\nabla_{\bbx_q}E_L^{\varepsilon}(\bbx_1,\dots,\bbx_N)}{w_q}\right).
\end{equation}

\section{Discrete gradient integrators for ODEs with dissipative and conservative quantities}
\label{sec:ODEform}

Motivated by the schemes \eqref{eq:schemeA2} and \eqref{eq:schemeL2}, we now consider a class of ODE system that our particle method belongs to, with the specific form
\begin{equation}
    \frac{\rd\bm X}{\rd t} = B(\bm  X)\nabla E(\bm X), \label{eq: dispODE}
\end{equation} where $\bm X(\cdot)\in \mathbb{R}^{k}$, $B(\cdot)$ is a $k\times k$, symmetric, negative semi-definite matrix and $E$ is a $C^1$ scalar function. We note that for our particle method, $\bm X$ will be a vector concatenated with all the particle positions, i.e., $\bm X^T=(\bm x_1^T,\dots,\bm x_N^T) \in \mathbb{R}^{dN}$ and the gradient operator $\nabla$ is understood as a concatenation of all the gradients with respect to each particle $\nabla^T = (\nabla_{\bm x_1}^T, \cdots, \nabla_{\bm x_N}^T)$. Moreover, the matrix $B(\bm X)$ of \eqref{eq: dispODE} can be viewed as a discretization of the bilinear form $\mathcal{B}[f]$ of \eqref{eq: bilinearForm} and the following lemma is the semi-discrete analog of \Cref{lem: bilinearForm} for the dissipative and conservative quantities.

\begin{lemma} \label{lem: dispConsQuan} $E(\bm X(t))$ is a nonincreasing function in $t$ along the solutions of \eqref{eq: dispODE}. In addition, if $I(\bm X)$ is a $C^1$ scalar function satisfying $\nabla I(\bm X) \perp B(\bm X)\nabla E(\bm X)$, then $I(\bm X)$ is a conserved function in $t$ along the solutions of \eqref{eq: dispODE}.
\end{lemma}
\begin{proof}
    By negative semi-definiteness of $B$, the solution of \eqref{eq: dispODE} satisfies
\begin{equation*}
    \frac{\rd E(\bm X)}{\rd t} = \nabla E(\bm X)^T\frac{\rd\bm X}{\rd t} = \nabla E(\bm X)^T B(\bm  X)\nabla E(\bm X) \leq 0.
\end{equation*}
Moreover, by hypothesis of $\nabla I$, the solution of \eqref{eq: dispODE} satisfies
    \begin{equation*}
    \frac{\rd I(\bm X)}{\rd t} = \nabla I(\bm X)^T\frac{\rd\bm X}{\rd t} = \nabla I(\bm X)^T B(\bm  X)\nabla E(\bm X) = 0.
\end{equation*}
\end{proof}

\subsection{Discrete gradient integration}

To discretize \eqref{eq: dispODE} while preserving the dissipative quantity $E$, a discrete gradient approximation $\overline{\nabla E}$ of $\nabla E$ can be used\footnote{In terms of accuracy in time, Itoh-Abe discrete gradient is first-order accurate, with midpoint discrete gradient and mean-value or average vector field discrete gradient being second-order accurate.}.

Consider the following discrete gradient scheme of \eqref{eq: dispODE}
\begin{equation}
    \frac{\bm X^{n+1} - \bm X^{n}}{\Delta t} = \tilde{B}(\bm X^{n+1}, \bm X^n)\overline{\nabla E}(\bm X^{n+1}, \bm X^n), \label{eq: discDispODE}
\end{equation} where $\tilde{B}(\bm X^{n+1}, \bm X^n)$ is a symmetric, negative semi-definite matrix \footnote{For second-order accuracy, we also require $\tilde{B}$ to satisfy $\tilde{B}(\bm X^{n+1}, \bm X^n) = \tilde{B}(\bm X^{n+1}, \bm X^n)$ \cite{HWL}, such as $\tilde{B}(\bm X^{n}, \bm X^{n+1}) = B\left(\frac{\bm X^{n+1}+\bm X^n}{2}\right)$.} approximating $B$ near $\bm X^{n+1}$ or $\bm X^n$. Moreover, to preserve a conservative quantity $I$ of \eqref{eq: dispODE} using \eqref{eq: discDispODE}, we further employ discrete gradient approximation for $\nabla I(\bm X)$, leading the following fully-discrete analog of \Cref{lem: bilinearForm} for the dissipative and conservative quantities. 
\begin{lemma}\label{lem: discDispConsQuan}
    $E(\bm X^{n+1})\leq E(\bm X^{n})$ for successive solutions $\bm X^{n}$ and $\bm X^{n+1}$ of \eqref{eq: discDispODE}. In additional, if $\overline{\nabla I}$ is a discrete gradient of $\nabla I$ satisfying $\overline{\nabla I}(\bm X^{n+1},\bm X^n) \perp$\\$\tilde{B}(\bm X^{n+1}, \bm X^n)\overline{\nabla E}(\bm X^{n+1}, \bm X^n)$, then $I(\bm X^{n+1}) = I(\bm X^n)$.
\end{lemma}
\begin{proof}
    It follows from \eqref{eq: discGradProp} that successive solutions of \eqref{eq: discDispODE} satisfy
\begin{equation*}
    E(\bm X^{n+1})-E(\bm X^{n}) = \Delta t\overline{\nabla E}(\bm X^{n+1}, \bm X^{n})^T \tilde{B}(\bm X^{n+1}, \bm X^n)\overline{\nabla E}(\bm X^{n+1}, \bm X^n) \leq 0. \label{eq: discEDisp}
\end{equation*}
    Moreover, by hypothesis of $\overline{\nabla I}$, the successive solutions of \eqref{eq: discDispODE} satisfy
    \begin{align*}
    I(\bm X^{n+1})-I(\bm X^{n}) &= \overline{\nabla I}(\bm X^{n+1},\bm X^n)^T(\bm X^{n+1}-\bm X^n) \\
    &= \Delta t\overline{\nabla I}(\bm X^{n+1},\bm X^n)^T\tilde{B}(\bm X^{n+1}, \bm X^n)\overline{\nabla E}(\bm X^{n+1}, \bm X^n) = 0.
\end{align*}
\end{proof}

\subsection{Example: aggregation-diffusion equation}
The particle method derived in \eqref{eq:schemeA2} can be written in the form of \eqref{eq: dispODE} as
\begin{equation}
\frac{\rd\bm X}{\rd t} = -W^{-1}\nabla E_A^\varepsilon (\bm X), \label{eq: aggDiffODE}
\end{equation} where $B(\bm X)=-W^{-1}$ with $W=\text{diag}(w_1,\dots, w_N)\otimes I_d$ being a diagonal particle weight matrix of size $dN\times dN$.
\begin{lemma}
    $E_A^\varepsilon(\bm X)$ is a dissipative quantity of \eqref{eq: aggDiffODE}.
\end{lemma}

\begin{proof}
    Since $B$ is symmetric, negative semi-definite, \eqref{eq: aggDiffODE} has a dissipative quantity $E_A^\varepsilon(\bm X)$ by Lemma \ref{lem: dispConsQuan}.
\end{proof}

\begin{lemma}
The discrete gradient scheme \eqref{eq: discDispODE} with $\tilde{B}(\bm X^{n+1},\bm X^n)=-W^{-1}$,
\begin{equation}
    \frac{\bm X^{n+1} - \bm X^{n}}{\Delta t} = -W^{-1}\overline{\nabla E_A^{\varepsilon}}(\bm X^{n+1}, \bm X^n), \label{eq: discAggDiff}
\end{equation} preserves the dissipative quantity $E_A^\varepsilon(\bm X)$.
\end{lemma}
\begin{proof}
    Since $\tilde{B}$ is symmetric, negative semi-definite, \eqref{eq: discAggDiff} preserves the dissipative quantity $E_A^\varepsilon(\bm X)$ by Lemma \ref{lem: discDispConsQuan}.
\end{proof}
\noindent Also, the mass $m(\bm X)=\frac{1}{d}\bm 1^TW\bm 1=\displaystyle \sum_{p=1}^N w_p$ is conserved, as $\nabla m=\bm 0 = \overline{\nabla m}$.

\subsection{Example: Landau equation}
The particle method derived in \eqref{eq:schemeL2} can be written in the form of \eqref{eq: dispODE} as
\begin{equation}
\frac{\rd\bm X}{\rd t} = -W^{-1}\mathcal{A}(\bm X)W^{-1}\nabla E_L^\varepsilon (\bm X), \label{eq: landauODE}
\end{equation} where $B(\bm X)=-W^{-1}\mathcal{A}(\bm X) W^{-1}$ with $W$ as the diagonal particle weight matrix as before, and $\mathcal{A}$ is the $dN\times dN$ matrix with $d\times d$ submatrix entries for $i,j=1,\dots,N$ given by
\begin{equation}
    [\mathcal{A}(\bm X)]_{ij} = \begin{cases}
        \displaystyle w_i\sum_{\substack{k=1\\ k\neq i}}^N w_k A(\bm x_i-\bm x_k), & \text{ if } i=j,\\
        -w_i w_j A(\bm x_i-\bm x_j), & \text{ if } i\neq j.
    \end{cases} \label{eq: Amatrix}
\end{equation}

\begin{lemma}
    $E_L^\varepsilon(\bm X)$ is a dissipative quantity of \eqref{eq: landauODE}. 
\end{lemma} 
\begin{proof}
    It suffices to show that $B(\bm X)$ is symmetric, negative semi-definite by Lemma \ref{lem: dispConsQuan}. First, since the $d\times d$ matrix $A$ is symmetric, $\mathcal{A}$ is symmetric and so $B(\bm X)=-W^{-1}\mathcal{A}(\bm X)W^{-1}$ is symmetric. Second, $\mathcal{A}$ is positive semi-definite, since $A(-\bm x)=A(\bm x)$ and for any $\bm Y^T=(\bm y_1^T,\dots, \bm y_N^T)\in \mathbb{R}^{dN}$
\begin{align*}
    &\bm Y^T\mathcal{A}(\bm X)\bm Y  = \sum_{i=1}^N \bm y_i^T\mathcal{A}(\bm X)_{ii} \bm y_i + \sum_{i=1}^N \sum_{i\neq j}^N \bm y_i^T\mathcal{A}(\bm X)_{ij} \bm y_j\\
    &= \sum_{i=1}^N \bm y_i^T\left(w_i \sum_{i\neq j}^N w_j  A(\bm x_i-\bm x_j)\right) \bm y_i+ \sum_{i=1}^N\sum_{i\neq j}^N \bm y_i^T\left(-w_i w_j A(\bm x_i-\bm x_j)\right) \bm y_j \\
    &= \sum_{i=1}^N \sum_{i<j}^N w_i w_j\left(\bm y_i^T A(\bm x_i-\bm x_j)\bm y_i + \bm y_j^T A(\bm x_j-\bm x_i)\bm y_j\right) \\
    &\hskip 4mm - \sum_{i=1}^N\sum_{i<j}^N w_i w_j 
    \left(\bm y_i^T A(\bm x_i-\bm x_j)\bm y_j + \bm y_j^T A(\bm x_j-\bm x_i)\bm y_i\right) \\
    &= \sum_{i=1}^N \sum_{i<j}^N w_i w_j (\bm y_i-\bm y_j)^T A(\bm x_i-\bm x_j) (\bm y_i - \bm y_j) \geq 0,
\end{align*} where the last steps follow from positive $w_i$ and $A$ being positive semi-definite. Thus, $B(\bm X)$ is negative semi-definite since for any $\bm Y$
\begin{equation*}
    \bm Y^TB(\bm X)\bm Y = -\bm (W^{-1}\bm Y)^T\mathcal{A}(\bm X)(W^{-1}\bm Y) \leq 0.
\end{equation*}
\end{proof}

Similar to the aggregation-diffusion equation, mass $m(\bm X)$ is trivially conserved. Moreover, momentum $\bm P(\bm X) =\bm 1^T\otimes I_d W\bm X = \displaystyle \sum_{p=1}^N w_p \bm x_p$ and kinetic energy $K(\bm X) =\frac{1}{2}\bm X^TW\bm X= \displaystyle \frac{1}{2}\sum_{p=1}^N w_p |\bm x_p|^2$ are conserved, where $\bm 1^T\otimes I_d$ is a $d\times dN$ matrix $\begin{pmatrix}I_d \cdots I_d\end{pmatrix}$.
\begin{lemma}
    $\bm P(\bm X)$ and $K(\bm X)$ are conserved quantities of \eqref{eq: landauODE}.
\end{lemma}
\begin{proof}
By Lemma \ref{lem: dispConsQuan}, it suffices to show that $\nabla \bm P(\bm X) \perp B(\bm X) \nabla E_L^\varepsilon(\bm X)$ and $\nabla K(\bm X)\perp B(\bm X) \nabla E_L^\varepsilon(\bm X)$. First note that $\nabla \bm P(\bm X) = (\bm 1^T\otimes I_d) W$ and since the matrix $\mathcal{A}$ of \eqref{eq: Amatrix} has zero blockwise column sum, $(\bm 1^T\otimes I_d)\mathcal{A}(\bm X) = \bm 0^T\otimes I_d$, implying
\begin{align*}
    \nabla \bm P(\bm X)^T B(\bm X) \nabla E_L^\varepsilon(\bm X) &= -\left(\bm 1^T\otimes I_d\right)W W^{-1} \mathcal{A}(\bm X)W^{-1}\nabla E_L^\varepsilon(\bm X) \\
    &= -\left(\bm 0^T\otimes I_d \right)W^{-1}\nabla E_L^\varepsilon(\bm X)= \bm 0.
\end{align*}
Also, it follows that $\nabla K(\bm X) = W\bm X$ and $\bm X^T \mathcal{A}(\bm X) = \bm 0^T$, leading to
\begin{align*}
    \nabla K(\bm X)^T B(\bm X) \nabla E_L^\varepsilon(\bm X) &=- \bm X^TW W^{-1}\mathcal{A}(\bm X)W^{-1}\nabla E_L^\varepsilon(\bm X) \\
    &=- \bm 0^TW^{-1}\nabla E_L^\varepsilon(\bm X) = 0.
\end{align*}
\end{proof}

\begin{lemma}
The discrete gradient scheme \eqref{eq: discDispODE} with $\tilde{B}(\bm X^{n+1},\bm X^n)=$\\$-W^{-1}\mathcal{A}\left(\frac{\bm X^{n+1}+\bm X^n}{2}\right)W^{-1}$,
\begin{equation}
    \frac{\bm X^{n+1} - \bm X^{n}}{\Delta t} = -W^{-1}\mathcal{A}\left(\frac{\bm X^{n+1}+\bm X^n}{2}\right)W^{-1}\overline{\nabla E_L^{\varepsilon}}(\bm X^{n+1}, \bm X^n), \label{eq: discLandau}
\end{equation} preserves the dissipative quantity $E_A^\varepsilon(\bm X)$ and the conservative quantities $\bm P$ and $K$.
\end{lemma}
%\sv{I think equation \ref{eq: discLandau} should use $\overline{\nabla E^{\varepsilon}_L}$ rather than $\overline{\nabla E^{\varepsilon}_A}$}
\begin{proof}
    Since $-\mathcal{A}$ is symmetric negative semi-definite and so is $-W^{-1}\mathcal{A}W^{-1}$, \eqref{eq: discLandau} preserves the dissipative quantity $E_L^\varepsilon(\bm X)$ by Lemma \ref{lem: discDispConsQuan}. Now to show that \eqref{eq: discLandau} preserves conservative quantities $\bm P$ and $K$, it suffices to verify by Lemma \ref{lem: discDispConsQuan} that both $\overline{\nabla \bm P}(\bm X^{n+1},\bm X^n),  \overline{\nabla K}(\bm X^{n+1},\bm X^n)$ are orthogonal to $\tilde{B}(\bm X^{n+1},\bm X^n)\overline{\nabla E_L^\varepsilon}(\bm X^{n+1}, \bm X^n)$.
    First, note that $\overline{\nabla \bm P}(\bm X^{n+1},\bm X^n) = (\bm 1^T\otimes I_d) W$ and $(\bm 1^T\otimes I_d)\mathcal{A}\left(\frac{\bm X^{n+1}+\bm X^n}{2}\right) = \bm 0^T\otimes I_d$, which shows
    \begin{align*}
        &\overline{\nabla \bm P}(\bm X^{n+1},\bm X^n)^T\tilde{B}(\bm X^{n+1},\bm X^n)\overline{\nabla E_L^\varepsilon}(\bm X^{n+1}, \bm X^n) \\
        &= -\left(\bm 1^T\otimes I_d\right)W W^{-1} \mathcal{A}\left(\frac{\bm X^{n+1}+\bm X^n}{2}\right)W^{-1}\overline{\nabla E_L^\varepsilon}(\bm X^{n+1},\bm X^n)\\
        &= -\left(\bm 0^T\otimes I_d \right)W^{-1}\overline{\nabla E_L^\varepsilon}(\bm X^{n+1},\bm X^n) = \bm 0.
    \end{align*}
    Second, it follows that $\overline{\nabla K}(\bm X^{n+1},\bm X^n) = W\left(\frac{\bm X^{n+1}+\bm X^n}{2}\right)$ and \\$\left(\frac{\bm X^{n+1}+\bm X^n}{2}\right)^T \mathcal{A}\left(\frac{\bm X^{n+1}+\bm X^n}{2}\right) = \bm 0^T$, leading to
\begin{align*}
    &\overline{\nabla K}(\bm X^{n+1},\bm X^n)^T \tilde{B}(\bm X^{n+1},\bm X^n) \overline{\nabla E_L^\varepsilon}(\bm X^{n+1},\bm X^n) \\
    &= -\left(\frac{\bm X^{n+1}+\bm X^n}{2}\right)^TW W^{-1}\mathcal{A}\left(\frac{\bm X^{n+1}+\bm X^n}{2}\right)W^{-1}\overline{\nabla E_L^\varepsilon}(\bm X^{n+1},\bm X^n)\\
    &= -\bm 0^TW^{-1}\overline{\nabla E_L^\varepsilon}(\bm X^{n+1},\bm X^n) = 0.
\end{align*}
\end{proof}

%%%%%%%%%%%%%%%%%%%%%%%%%%
\section{Numerical examples}
\label{sec:Examples}
In this section, we present several numerical examples to verify the theoretical results of \Cref{sec:ODEform}.  Examples \ref{example:heat eq}, \ref{example:porous medium eq}, \ref{example: linear fokker-planck eq}, \ref{example: non-local fokker-planck eq} correspond to the aggregation-diffusion equation \eqref{eq:GFP} with one spatial dimension and are similar to the examples given in \cite{BCH20}. 
 Examples \ref{example: Landau eq Maxwellian} and \ref{example: Landau eq Coulomb} correspond to the Landau equation \eqref{eq:Landau1} with two spatial dimensions and are similar to the examples given in \cite{CHWW20}.  To compare the particle solution $f^N$ with an analytical solution $f$, the particle solution is convolved with the Gaussian mollifier function
 \begin{equation}\label{eq:blob solution}
 \begin{aligned}
    f^{N}_{\varepsilon}(t,\bbx) :=& \;(\varphi_{\varepsilon}*f^N)(t,\bbx) = \sum_{p=1}^Nw_p\varphi_{\varepsilon}(\bbx-\bbx_p(t)).
 %   \varphi_{\varepsilon}(\bbx) =& \; \frac{1}{(2\pi \varepsilon)^{d/2}}\exp\left(-\frac{|\bbx|^2}{2\varepsilon}\right).
\end{aligned}
\end{equation}
Using the initialization, computational domain, and cell centers $\{\bbx_i\}$ described at the beginning of \Cref{sec:Particle method}, the  $L^{p}$ and $L^{\infty}$ errors are defined as
\begin{equation*}
\|f^{N}_{\varepsilon} - f\|^p_{L^p} = \sum_{i}h^d|f^N_{\varepsilon}(\bbx_{i}) - f(\bbx_{i})|^p, \quad \|f^{N}_{\varepsilon} - f\|_{L^{\infty}} = \max_{i}|f^N_{\varepsilon}(\bbx_{i}) - f(\bbx_{i})|.
\end{equation*}
When choosing the regularization parameter $\varepsilon$, we are motivated by the success of the numerical examples presented in \cite{CCP19} and \cite{CHWW20}.  For all the examples presented in this paper, we choose $\varepsilon = 0.64h^{1.98}$.  Additionally, for each of the following examples, we track the time evolution of the conserved quantities.  Specifically, mass, momentum, and kinetic energy of the particle solution are given by 
\begin{equation}
    m(\bm X^n)=\sum_{p = 1}^N w_p, \quad P(\bm X^n) =\sum_{p=1}^N w_p\bbx^n_p, \quad \mbox{and} \quad K(\bm X^n) = \frac{1}{2}\sum_{p=1}^N w_p|\bbx^n_p|^2.
\end{equation}
We omit plots of the time evolution of mass, since the particle weights $\{w_p\}$ remain as constant. The energy $E_{A/L}^{\varepsilon}(\bm X^n)$ is defined using \eqref{Hreg} evaluated at $f^N$, given by
\begin{equation}
E^{\varepsilon}_A(\bm X^n) = \int_{\mathbb{R}^d}H\left(\sum_{p=1}^N w_p\varphi_{\varepsilon}(\bbx - \bbx_p^n)\right)\rd{\bbx} + \sum_{p=1}^Nw_pV(\bbx_p^n) + \frac{1}{2}\sum_{p,q=1}^Nw_qw_pW(\bbx_p^n - \bbx_q^n),
\end{equation}
where $H(f) = f\log f$ or $\frac{1}{m-1}f^m$, and
\begin{equation}
E^{\varepsilon}_L(\bm X^n) = \int_{\mathbb{R}^d}\sum_{p=1}^Nw_p\varphi_{\varepsilon}(\bbx - \bbx_p^n)\log{\left(\sum_{q=1}^Nw_q\varphi_{\varepsilon}(\bbx - \bbx_q^n)\right)}\rd{\bbx}.
\end{equation}
The integrals above are approximated by the midpoint rule using grid $\{\bbx_i\}$.

Scheme \eqref{eq: discAggDiff} can be written in terms of each particle 
\begin{equation}\label{eq:discaggdiff1}
    \frac{\bbx_p^{n+1}-\bbx_p^n}{\Delta t}=-\frac{1}{w_p}\overline{\nabla_{\bbx_p}E_A^{\varepsilon}}(\bm X^{n+1},\bm X^n),
\end{equation}
and the same can be done for scheme \eqref{eq: discLandau}
\begin{equation}\label{eq:discLand1}
   \frac{\bbx_p^{n+1}-\bbx_p^n}{\Delta t}\\ 
    = -\sum_{q=1}^N w_q A(\overline{\bbx_p^n}-\overline{\bbx_q^n}) \left(\frac{\overline{\nabla_{\bbx_p}E_L^{\varepsilon}}(\bm X^{n+1},\bm X^n)}{w_p}-\frac{\overline{\nabla_{\bbx_q}E_L^{\varepsilon}}(\bm X^{n+1},\bm X^n)}{w_q}\right),
\end{equation}
where $\overline{\bbx_p^n} =(\bbx_p^n+\bbx_p^{n+1})/2$.
As mentioned in \Cref{sec:Intro}, the following examples make use of the mean-value discrete gradient,
\begin{equation}\label{eq:AVF DG}
\overline{\nabla_{\bbx_p} E_{A/L}^{\varepsilon}}(\bm X^{n+1},\bm X^n):=\int_0^1 \nabla_{\bbx_p} E_{A/L}^{\varepsilon}\left(\bX^n+s(\bX^{n+1}-\bX^n)\right)\rd{s},
\end{equation}
and by \Cref{cor}, $\nabla_{\bbx_p}E^{\varepsilon}_{A/L}(\bm X^n) = w_p\nabla_{\bbx}\frac{\delta E^{\varepsilon}_{A/L}}{\delta f}(\bbx_p^n)$ and the integral in \eqref{eq:AVF DG} is approximated using a four point Gauss-Legendre quadrature.

The fixed point iteration, with a forward Euler step as the initial guess, is used to approximate the solution to the nonlinear system resulting from the discrete gradient discretization \eqref{eq:discaggdiff1} and \eqref{eq:discLand1}.  The stopping criterion is when the relative error from two consecutive iterations is less than $10^{-15}$.  For each example, the average number of iterations and the maximum amount of iterations required to meet the stopping criterion is listed.  It is important to note that all the examples except Example \ref{example:porous medium eq} have a maximum number of iterations that is much larger than the average amount of iterations, and that the maximum amount of iterations occurs at the initial time or shortly after. We compared the results of using a forward Euler initial guess to a fourth order Runge-Kutta initial guess but did not see much improvement on the number of iterations. We also explored other options for solving the nonlinear system.  The MATLAB function \lstinline{fsolve} is slow to converge even for a small amount of particles.  Newton's method is another option, but the formulation of the Jacobian for the Landau equation is both complicated and computationally expensive.  Finally, a Newton Krylov solver is another option and is used in \cite{ML11} when solving the collisionless Vlasov-Maxwell system. We defer the investigation of Newton type solvers to future work.

\subsection{Heat equation} 
\label{example:heat eq}
The first example is the 1D heat equation 
\begin{equation}
\partial_t f = \partial_{xx}f,
\end{equation}
which is the aggregation-diffusion equation \eqref{eq:GFP} with 
\begin{equation*}
    H(f) = f\log{f}, \quad V(x) = 0, \quad W(x) = 0.
\end{equation*}
The analytical solution corresponding to a point source is 
\begin{equation}\label{eq:heat equation solution}
    \Phi(t,x) = (4\pi t)^{-\frac{1}{2}}\exp{\left(-\frac{x^2}{4t} \right)}.
\end{equation}
We consider $t \in [2,3]$ with the initial condition $f(2,x) = \Phi(2,x)$ and a computational domain of $[-15,15]$. The particle solution using $\Delta t = 0.01$ and $M =60$, $70$, $80$, $90$, and $100$ is compared to the analytical solution \eqref{eq:heat equation solution}.  The left plot of Figure \ref{fig:heat eq figures} shows the decay of energy of the particle solution.  The right plot of Figure \ref{fig:heat eq figures} is a log-log plot of the $L^1,L^2,$ and $L^{\infty}$ errors vs the cell size $h$ at the final time $t = 3$ and shows that the particle method is approximately second order accurate.  Figure \ref{fig:heat eq num iter} shows how many iterations the fixed point method needs to meet the stopping criterion along time.  Table \ref{table:heat eq iter} shows the mean and maximum number of iterations for different $M$. 

\begin{figure}[!ht]
    \centering
    \includegraphics[width = .49\textwidth]{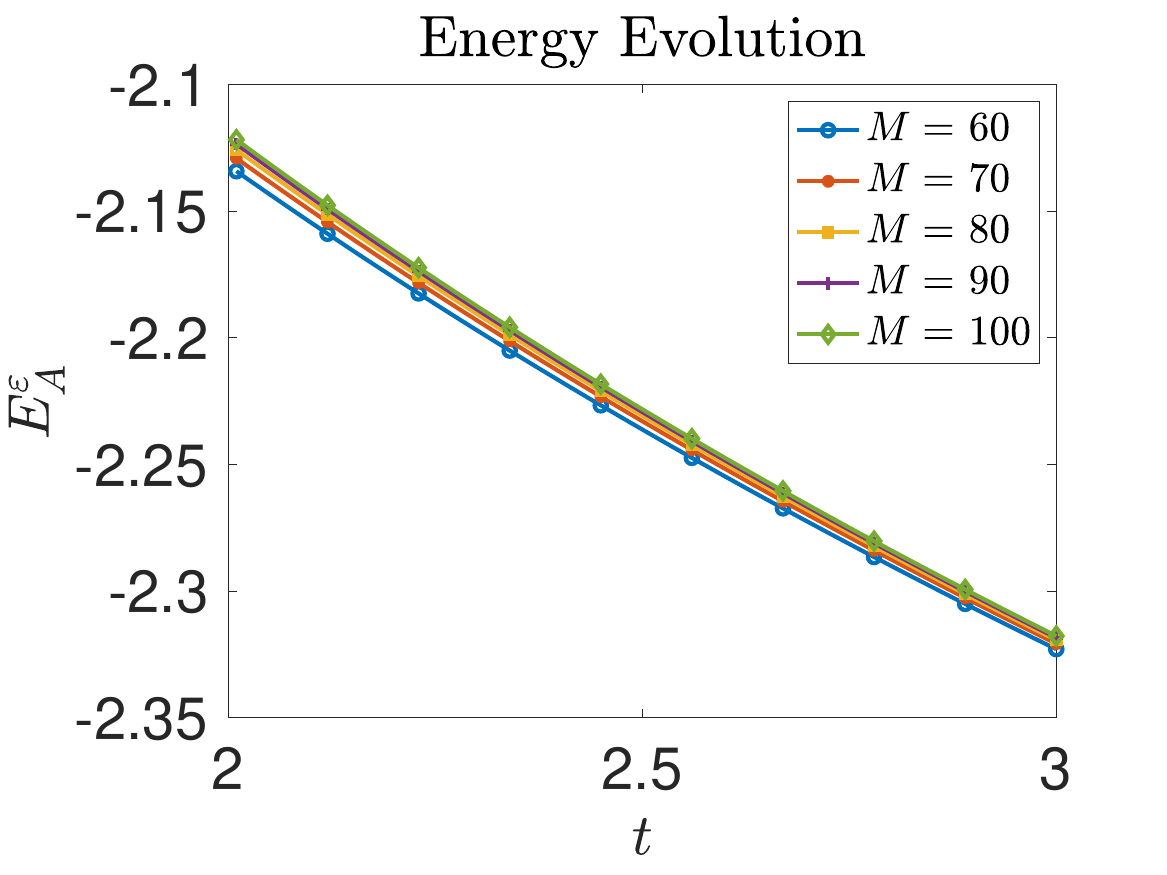}
    \includegraphics[width = .49\textwidth]{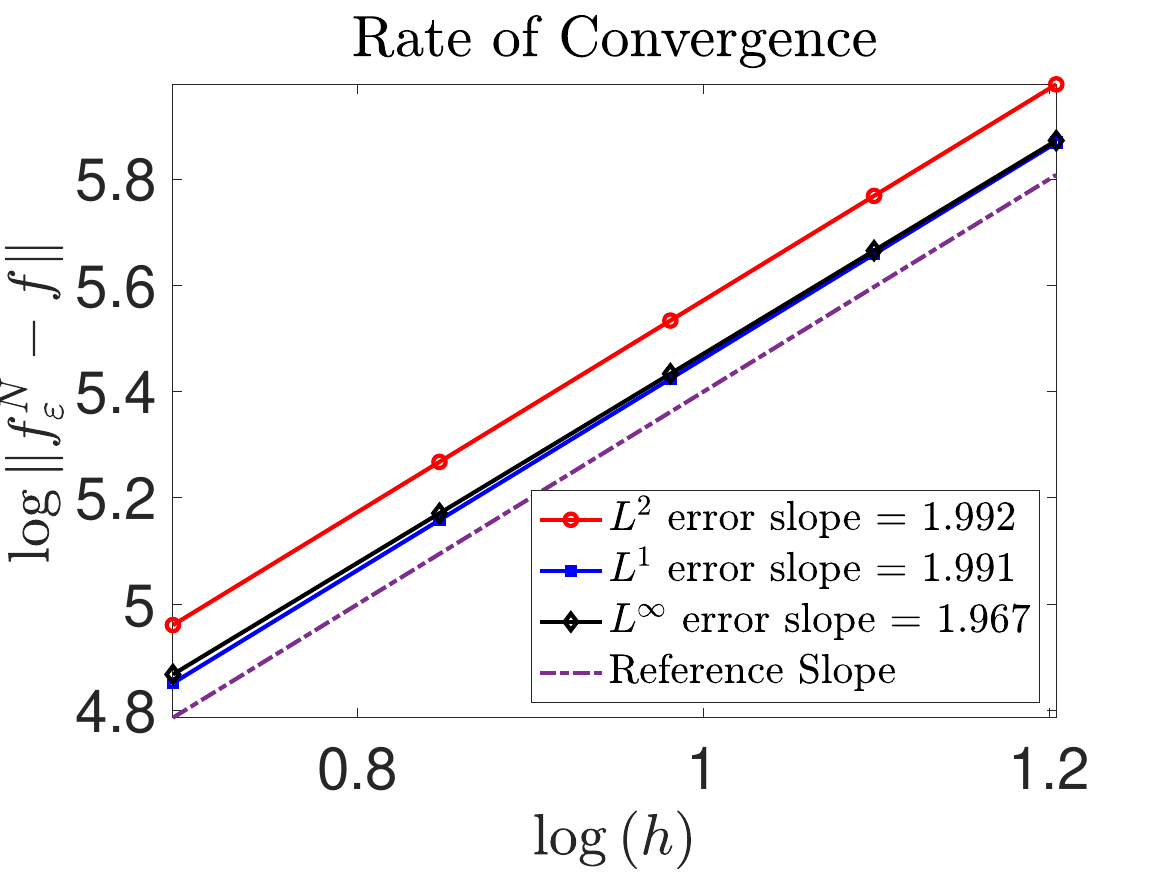}
    \caption{Example \ref{example:heat eq}: Left: Time evolution of the energy of the particle solution for different values of $M$. 
    Right: $L^1, L^2$ and $L^{\infty}$ errors vs cell size $h$.}
    \label{fig:heat eq figures}
\end{figure}

\begin{minipage}{.95\linewidth}
  \begin{minipage}[b]{.45\linewidth}
    \centering
\includegraphics[width=.99\textwidth]{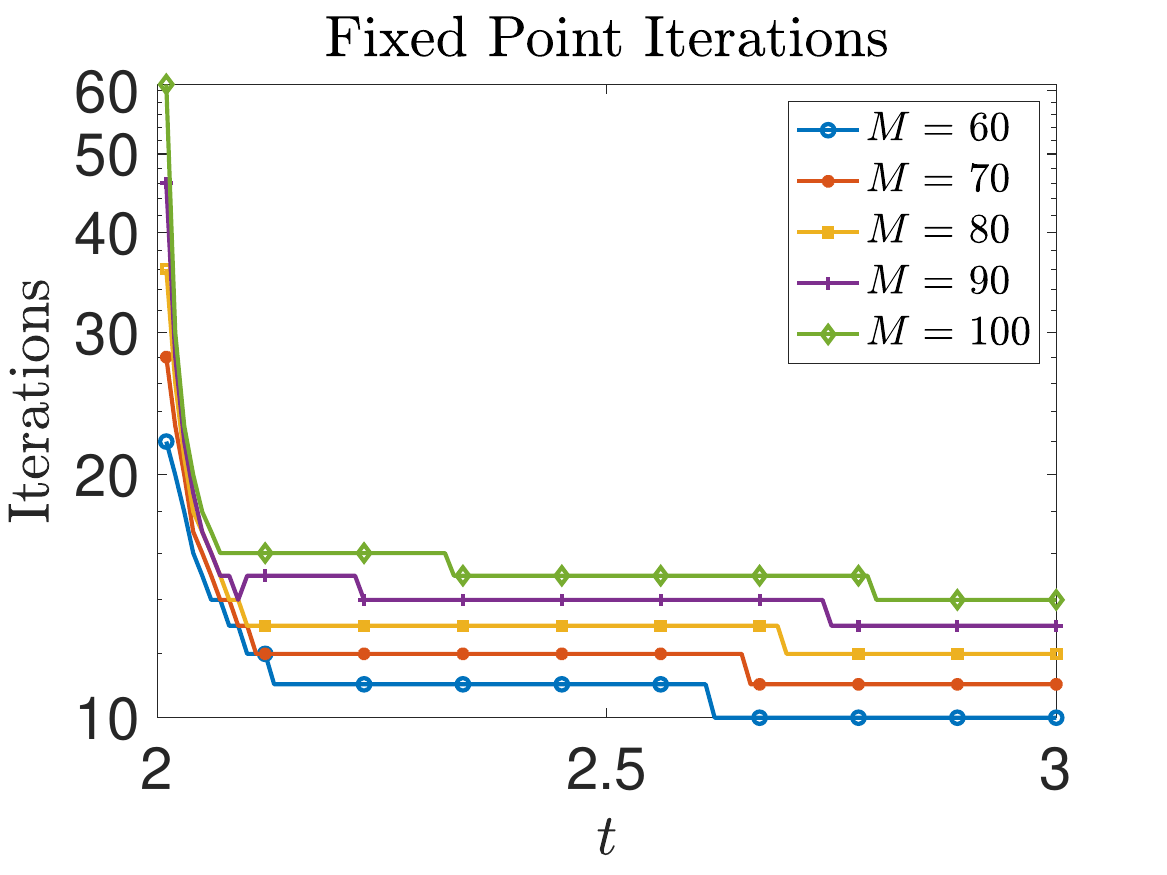}
    \captionof{figure}{Example \ref{example:heat eq}: Number of fixed point iterations required to meet the stoping criterion along time.}%
  \label{fig:heat eq num iter}% \caption{Figure caption}
  \end{minipage} ~~
  \begin{minipage}[b]{.5\linewidth}
    \centering
    \begin{tabular}{|c|c|c|} \hline
  $M$ & mean number & max number\\ 
      & of iterations  & of iterations\\
      \hline
  $60$ & 11.10 & 22 \\
  \hline
  $70$ & 12.18 & 28 \\ 
  \hline
  $80$ & 13.29 & 36 \\
  \hline
  $90$ & 14.53 & 46 \\
  \hline
  $100$ & 15.84 & 61\\
  \hline
  \end{tabular}
    \captionof{table}{Example \ref{example:heat eq}: The average number and maximum number of fixed point iterations required to meet the stopping criterion for different values of $M$.  The time step is $\Delta t = 0.01$ for all values of $M$.}% 
  \label{table:heat eq iter}
  \end{minipage}
\end{minipage}

\subsection{Porous medium equation}
\label{example:porous medium eq}
This example concerns the 1D porous medium equation 
\begin{equation}\label{eq: porous medium eq}
\partial_t f = \partial_{xx}(f^m), \quad m>1,   
\end{equation}
which is the aggregation-diffusion equation \eqref{eq:GFP} with 
\begin{equation*}
    H(f)=\frac{1}{m-1}f^m, \quad V(x) = 0,\quad W(x) = 0.
\end{equation*}
We compare our particle solution with the analytical Barenblatt solution \cite{B52}, 
\begin{equation}\label{eq:Porous Medium eq solution}
\Psi(t,x) = \frac{1}{t^{\alpha}}\psi\left(\frac{|x|}{t^{\alpha}} \right),
\end{equation}
where $\psi(\xi) = \left(K - \kappa \xi^2\right)_{+}^{1/(m-1)}$ for $\alpha = 1/(m+1)$, $\gamma = 1/(m-1) + 1/2$, $\kappa = \alpha(m-1)/(2m)$ and $(\cdot)_{+} = \max{\left\{\cdot,0\right\}}$.  The normalization constant $K > 0$ is related to the total mass $ a(m)K^{\gamma}$, where
\begin{equation*}
    a(m) = \left(\frac{2\pi m}{\alpha(m-1)}\right)^{\frac{1}{2}}\frac{\Gamma\left(\frac{m}{m-1}\right)}{\Gamma{\left(\frac{m}{m-1}+\frac{1}{2}\right)}},
\end{equation*}
and $\Gamma$ is the Gamma function. 

We consider the porous medium equation \eqref{eq: porous medium eq} with $m = 3/2$ for $t \in [2,3]$ with the initial condition $f(2,x) = \Psi(2,x)$ and $K=1$. The computational domain is $[-8,8]$.
The particle solution using $\Delta t = 0.01$ and $M = 60, 70, 80, 90,$ and $100$ is compared to the analytical solution \eqref{eq:Porous Medium eq solution}. The left plot of Figure \ref{fig:porous medium} shows the decay of energy of the particle solution.  The right plot of Figure \ref{fig:porous medium} is a log-log plot of the $L^1,L^2,$ and $L^{\infty}$ errors vs the cell size $h$ at the final time $t = 3$ and shows that the particle method is approximately second order accurate. Figure \ref{fig:porous medium num iter} shows how many iterations the fixed point method needs to meet the stopping criterion along time. Table \ref{table:porous medium iter} shows the mean and maximum number of iterations for different values of $M$.

\begin{figure}[!ht]
    \centering
    \includegraphics[width = .49\textwidth]{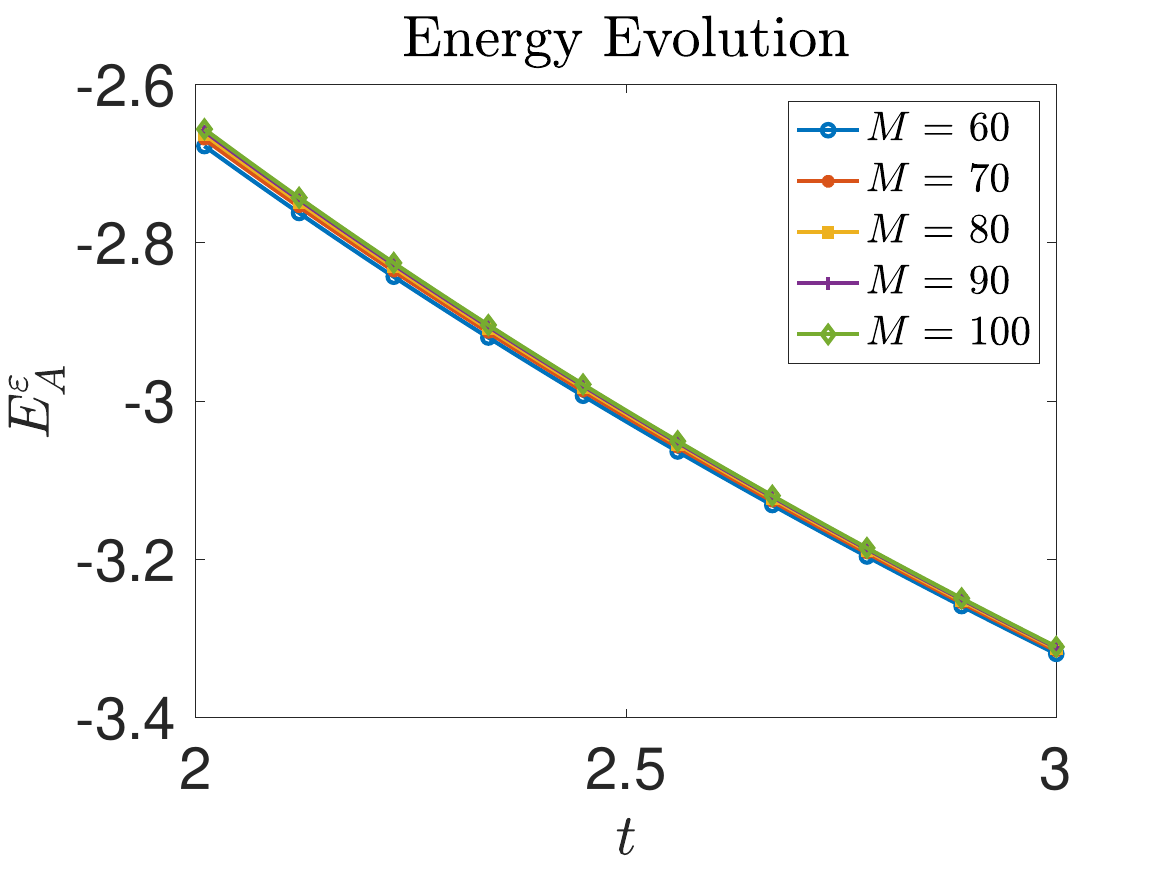}
    \includegraphics[width = .49\textwidth]{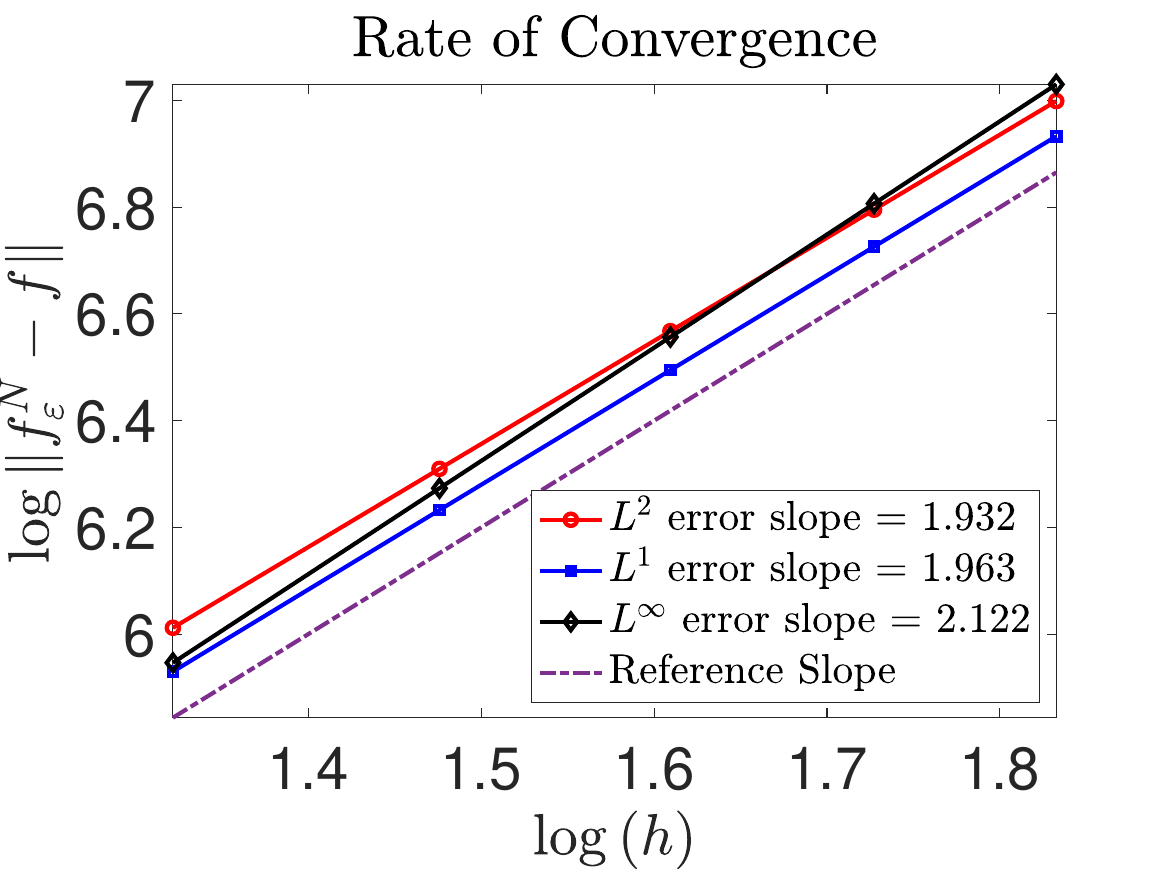}
    \caption{Example \ref{example:porous medium eq}. Left: Time evolution of the energy of the particle solution for different values of $M$. 
 Right: $L^1, L^2$ and $L^{\infty}$ errors vs cell size $h$.}
 \label{fig:porous medium}
\end{figure}

\begin{minipage}{0.95\linewidth}
  \begin{minipage}[b]{.45\linewidth}
    \centering
    \includegraphics[width=.99\textwidth]{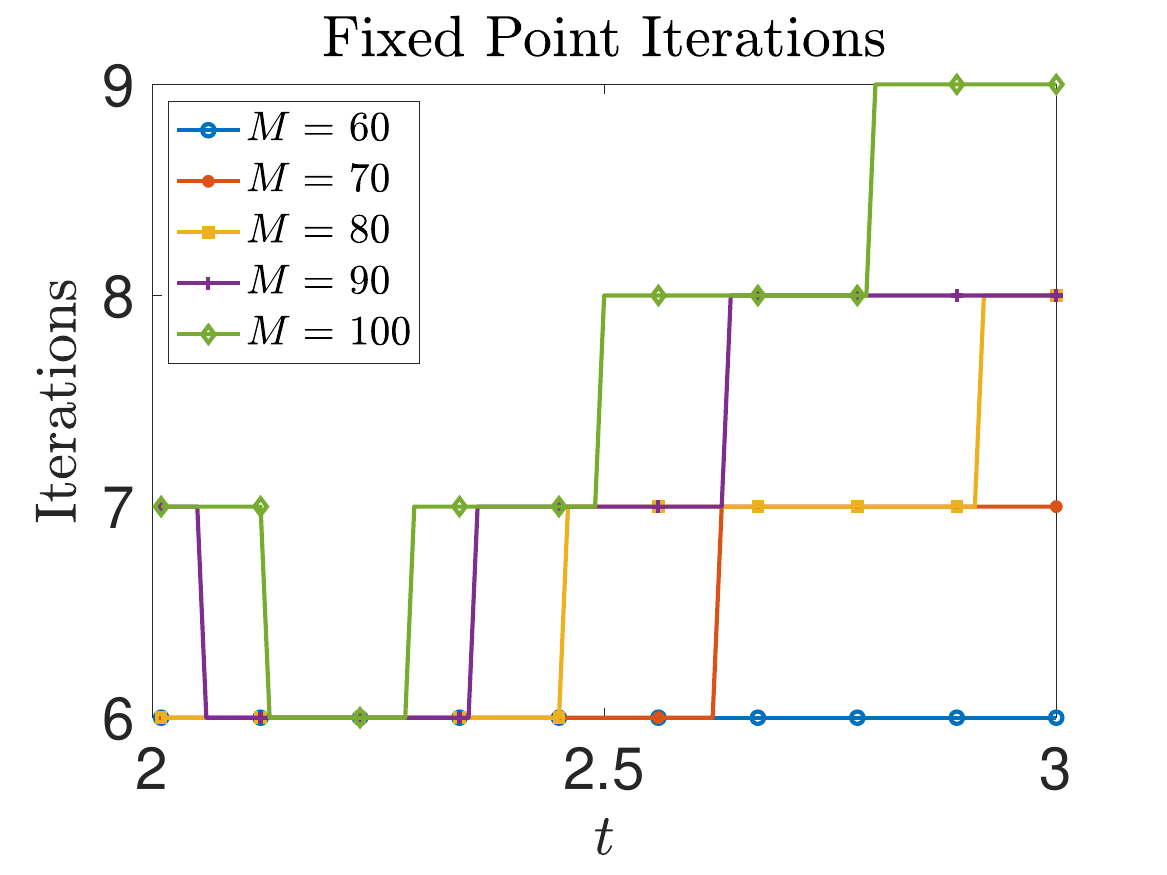}
    \captionof{figure}{Example \ref{example:porous medium eq}: Number of fixed point iterations required to meet the stoping criterion along time.}
 \label{fig:porous medium num iter}% \caption{Figure caption}
  \end{minipage} ~~
  \begin{minipage}[b]{.5\linewidth}
    \centering
    \begin{tabular}{|c|c|c|} \hline
  $M$ & mean number & max number\\ 
      & of iterations  & of iterations\\
      \hline
  $60$ & 6.00 & 6 \\
  \hline
  $70$ & 6.38 & 7 \\ 
  \hline
  $80$ & 6.64 & 8 \\
  \hline
  $90$ & 7.07 & 8 \\
  \hline
  $100$ & 7.56 & 9\\
  \hline
  \end{tabular}
    \captionof{table}{Example \ref{example:porous medium eq}: The average number and maximum number of fixed point iterations required to meet the stopping criterion for different values of $M$.  The time step is $\Delta t = 0.01$ for all values of $M$.}%
  \label{table:porous medium iter}
  \end{minipage}
\end{minipage}

\subsection{Linear Fokker-Planck equation}
\label{example: linear fokker-planck eq}
This example concerns the 1D linear Fokker Planck equation
\begin{equation}\label{eq:1D Linear Fokker Planck}
       \partial_t f=\partial_{xx}f +\partial_x(xf),
\end{equation}
which is the aggregation-diffusion equation \eqref{eq:GFP} with 
\begin{equation*}
    H(f) = f\log{f},\quad V(x) = \frac{x^2}{2}, \quad W(x) = 0.
\end{equation*}
We compare our particle solution to the analytical solution 
\begin{equation}\label{eq:1D Linear Fokker Planck solution}
        \Upsilon(t,x) = \left(2\pi\left(1-e^{-2t}\right)\right)^{-1/2}\exp\left(-\frac{x^2}{2\left(1-e^{-2t} \right)}\right).
\end{equation}
We consider $t \in [0.5,1]$ with the initial condition $f(0.5,x) = \Upsilon(0.5,x)$. The computational domain is $[-5,5]$.  The particle solution using $\Delta t = 0.01/10$ and $M = 60, 70, 80, 90,$ and $100$ is compared to the analytical solution \eqref{eq:1D Linear Fokker Planck solution}.  

\begin{figure}[!ht]
    \centering
    \includegraphics[width = .49\textwidth]{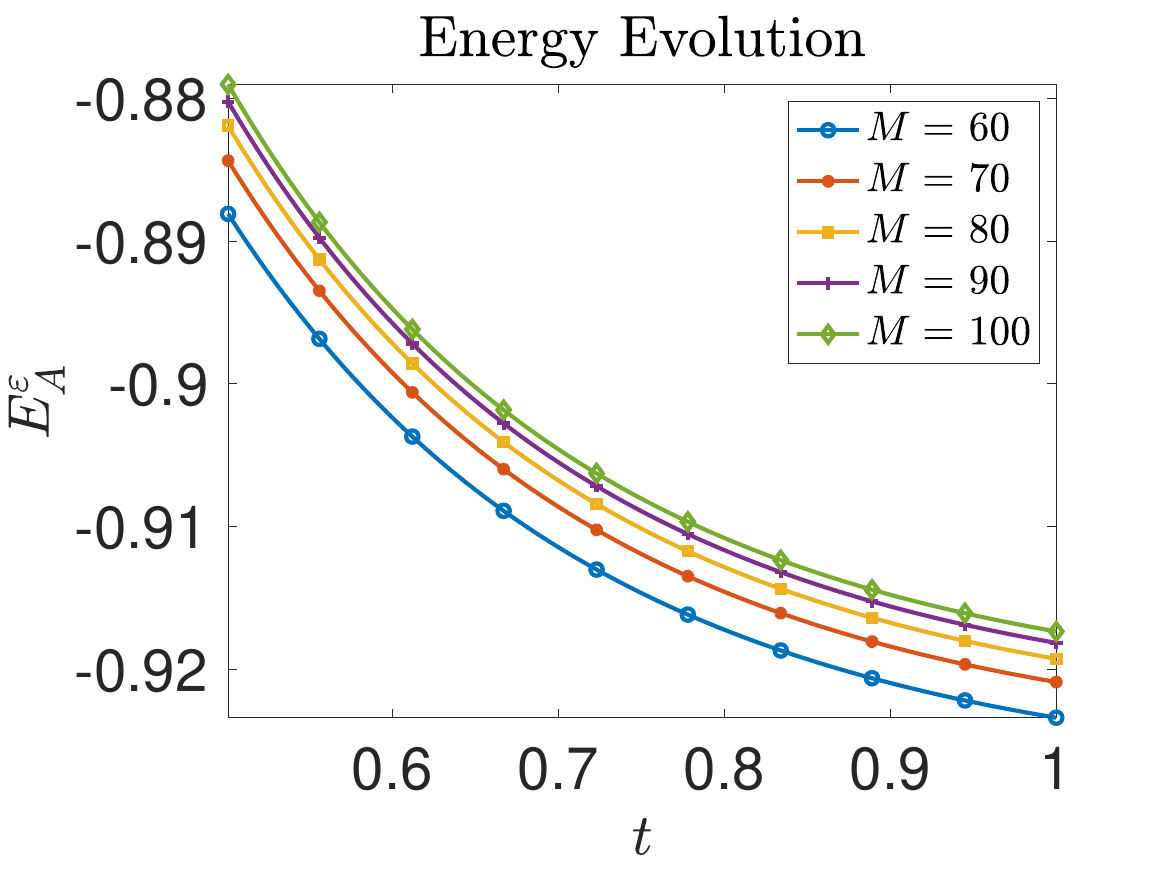}
    \includegraphics[width = .49\textwidth]{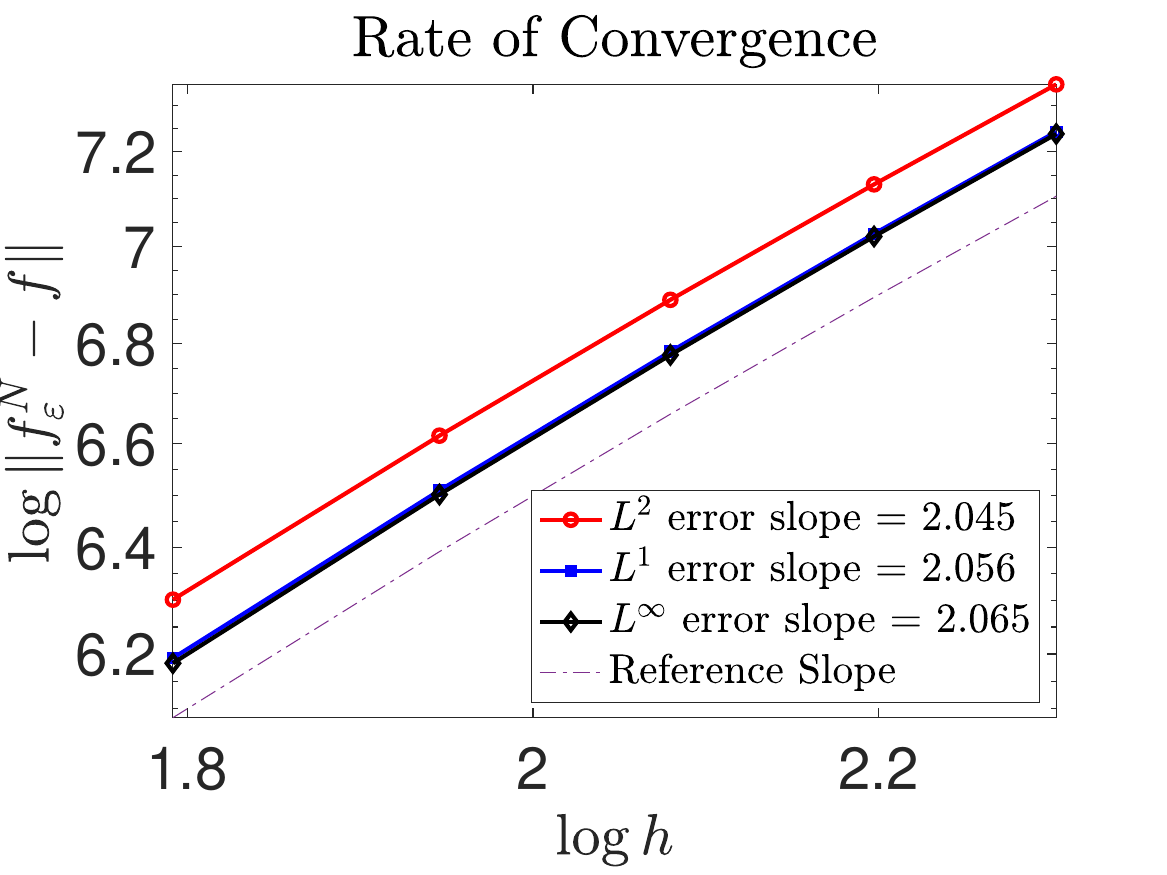}
    \caption{Example \ref{example: linear fokker-planck eq}: Left: Time evolution of the energy of the particle solution for different values of $M$. 
 Right: $L^1, L^2$ and $L^{\infty}$ errors vs cell size $h$.}
 \label{fig:1D linear fokker planck}
\end{figure}

The left plot of Figure \ref{fig:1D linear fokker planck} shows the decay of energy of the particle solution.   The right plot of Figure \ref{fig:1D linear fokker planck} is a log-log plot of the $L^1,L^2,$ and $L^{\infty}$ errors vs the cell size $h$ at the final time $t = 1$ and shows that the particle method is approximately second order accurate. Figure \ref{fig:fokker planck linear num iter} shows how many iterations the fixed point method needs to meet the stopping criterion along time.  Table \ref{table:linear fokker planck iter} shows the mean and maximum number of iterations for different values of $M$. \\

\begin{minipage}{0.95\linewidth}
  \begin{minipage}[b]{.45\linewidth}
    \centering
    \includegraphics[width=.99\textwidth]{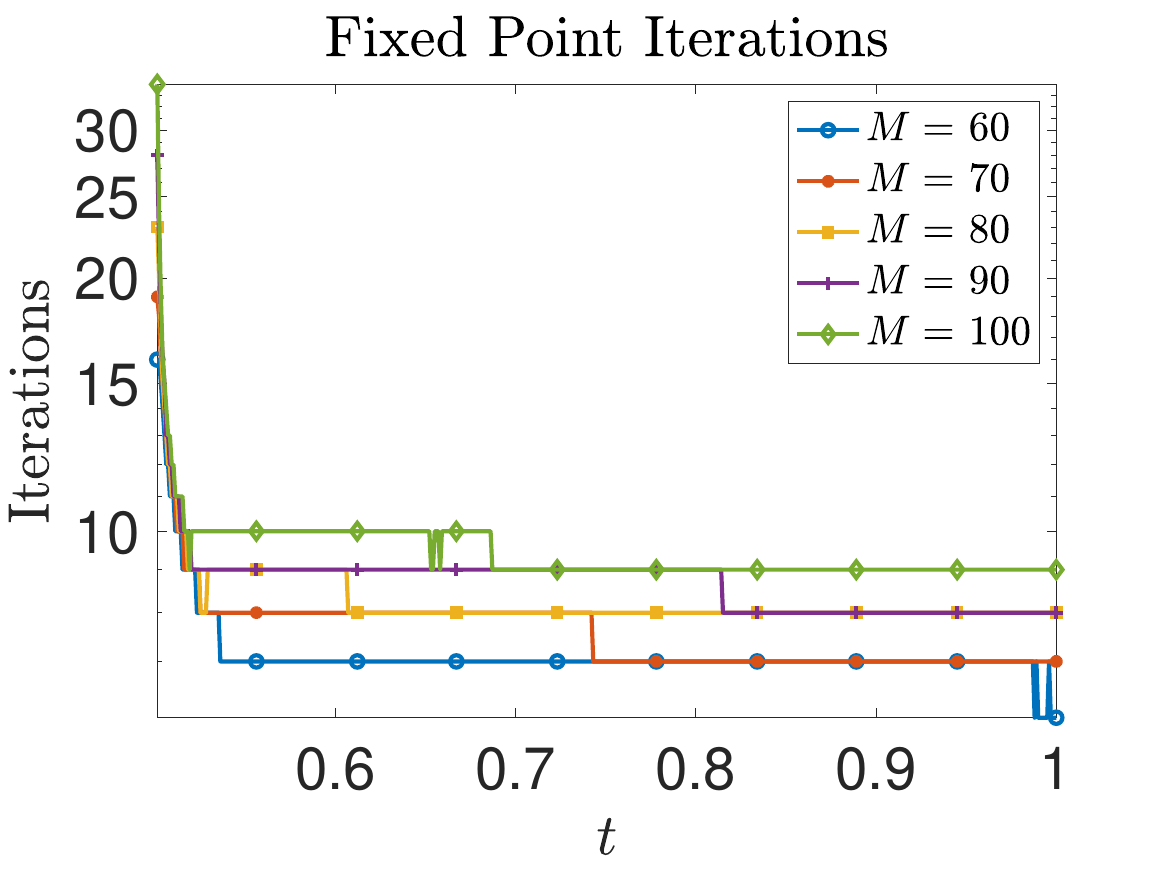}
    \captionof{figure}{Example \ref{example: linear fokker-planck eq}: Number of fixed point iterations required to meet the stoping criterion along time.}%
  \label{fig:fokker planck linear num iter}
  \end{minipage} ~~
  \begin{minipage}[b]{.5\linewidth}
    \centering
    \begin{tabular}{|c|c|c|} \hline
  $M$ & mean number & max number\\ 
      & of iterations  & of iterations\\
      \hline
  $60$ & 7.18  & 16 \\
  \hline
  $70$ & 7.65 & 19 \\ 
  \hline
  $80$ & 8.34 & 23 \\
  \hline
  $90$ & 8.79 & 28 \\
  \hline
  $100$ & 9.52 & 34\\
  \hline
  \end{tabular}
    \captionof{table}{Example \ref{example: linear fokker-planck eq}: The average number and maximum number of fixed point iterations required to meet the stopping criterion for different values of $M$.  The time step is $\Delta t = 0.01/10$ for all values of $M$.}%
  \label{table:linear fokker planck iter}
  \end{minipage}
\end{minipage}

\subsection{Non-local Fokker-Planck equation}
\label{example: non-local fokker-planck eq}
The Fokker Planck equation \eqref{eq:1D Linear Fokker Planck} can be written equivalently as the aggregation-diffusion equation \eqref{eq:GFP} with a nonzero interaction potential rather than a nonzero external potential
\begin{equation*}
    H(f) = f\log{f},\quad V(x) = 0, \quad W(x) = \frac{x^2}{2}.
\end{equation*}
We use this example to showcase how the method handles the convolution term $W*f$. The setup is exactly the same as the previous example. The results are shown in Figures \ref{fig:1D_non_local_fokker_planck}, \ref{fig:fokker planck non local num iter} and Table \ref{table:non local fokker planck iter}.

\begin{figure}[!ht]
    \centering
    \includegraphics[width = .49\textwidth]{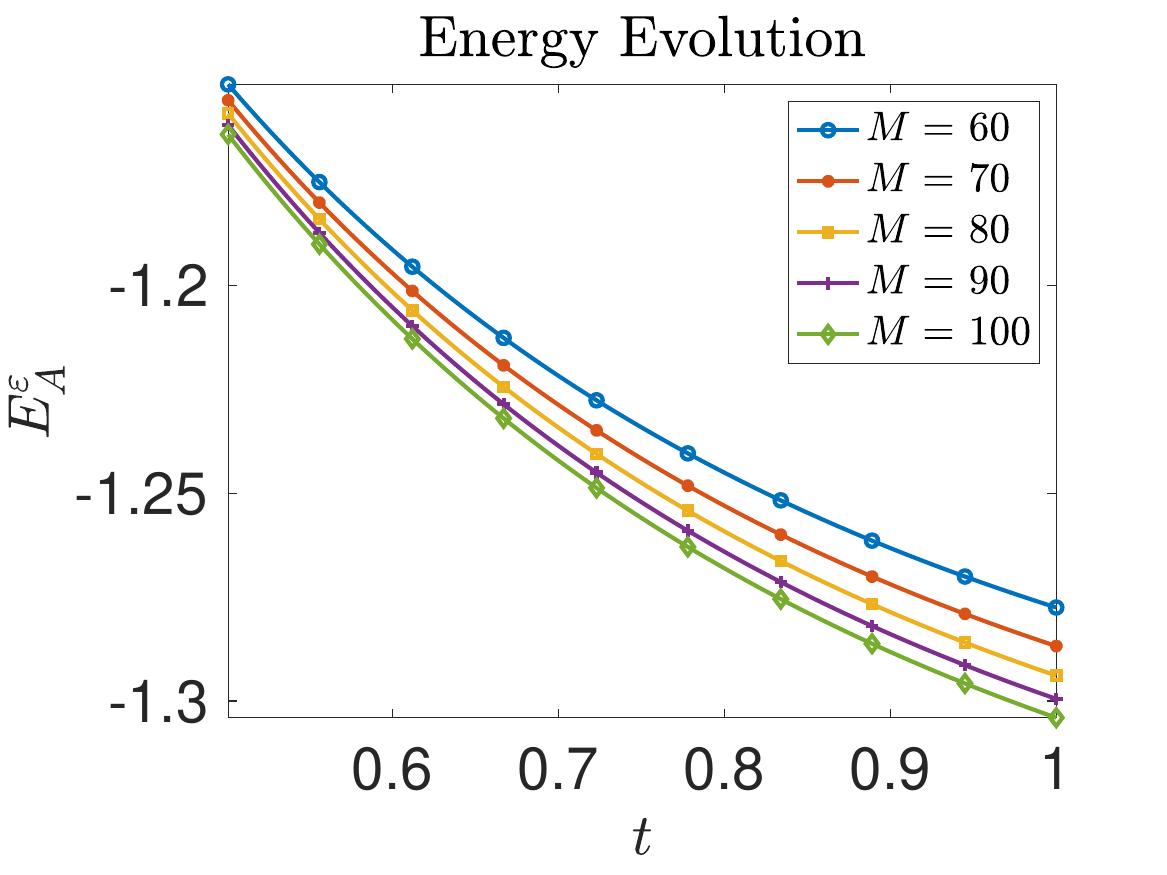}
    \includegraphics[width = .49\textwidth]{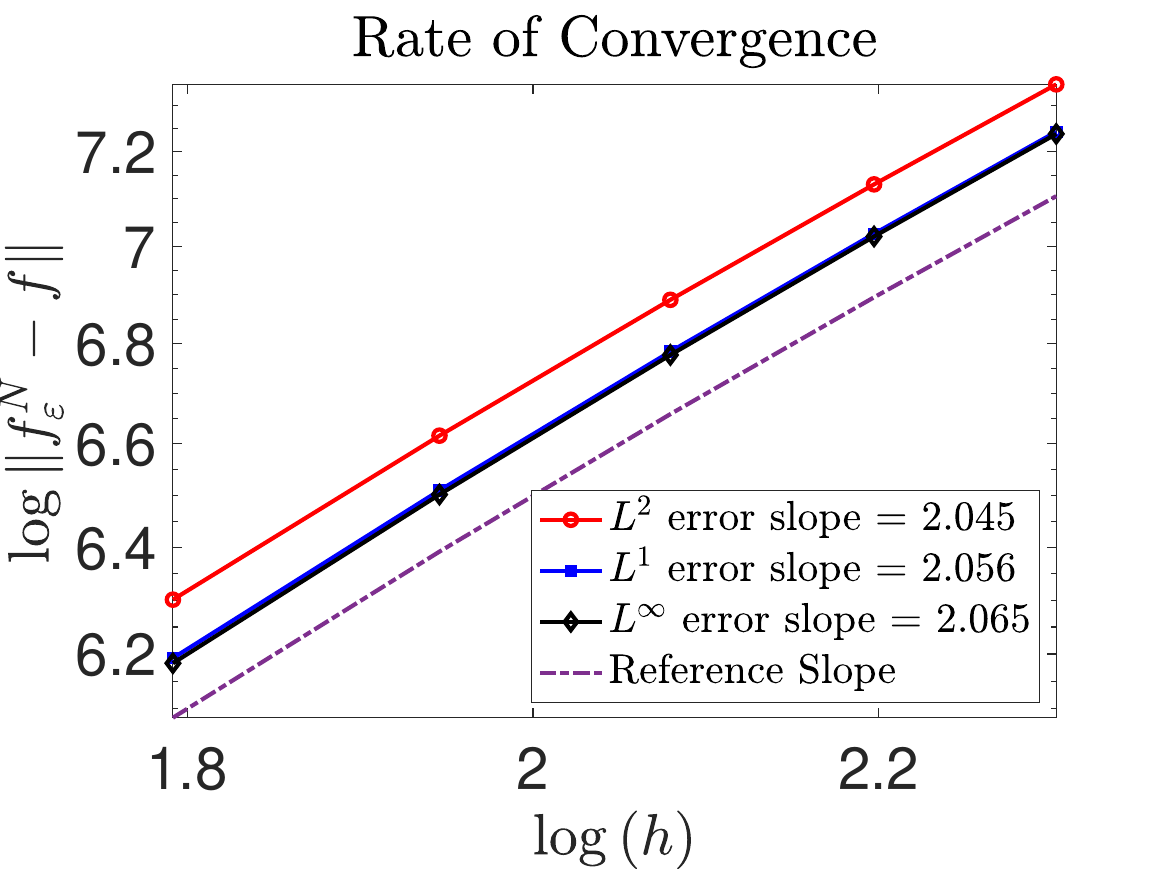}
    \caption{Example \ref{example: non-local fokker-planck eq}: Left: Time evolution of the energy of the particle solution for different values of $M$. 
 Right: $L^1, L^2$ and $L^{\infty}$ errors vs cell size $h$.}
 \label{fig:1D_non_local_fokker_planck}
\end{figure}

\begin{minipage}{0.95\linewidth}
  \begin{minipage}[b]{.45\linewidth}
    \centering
    \includegraphics[width=.99\textwidth]{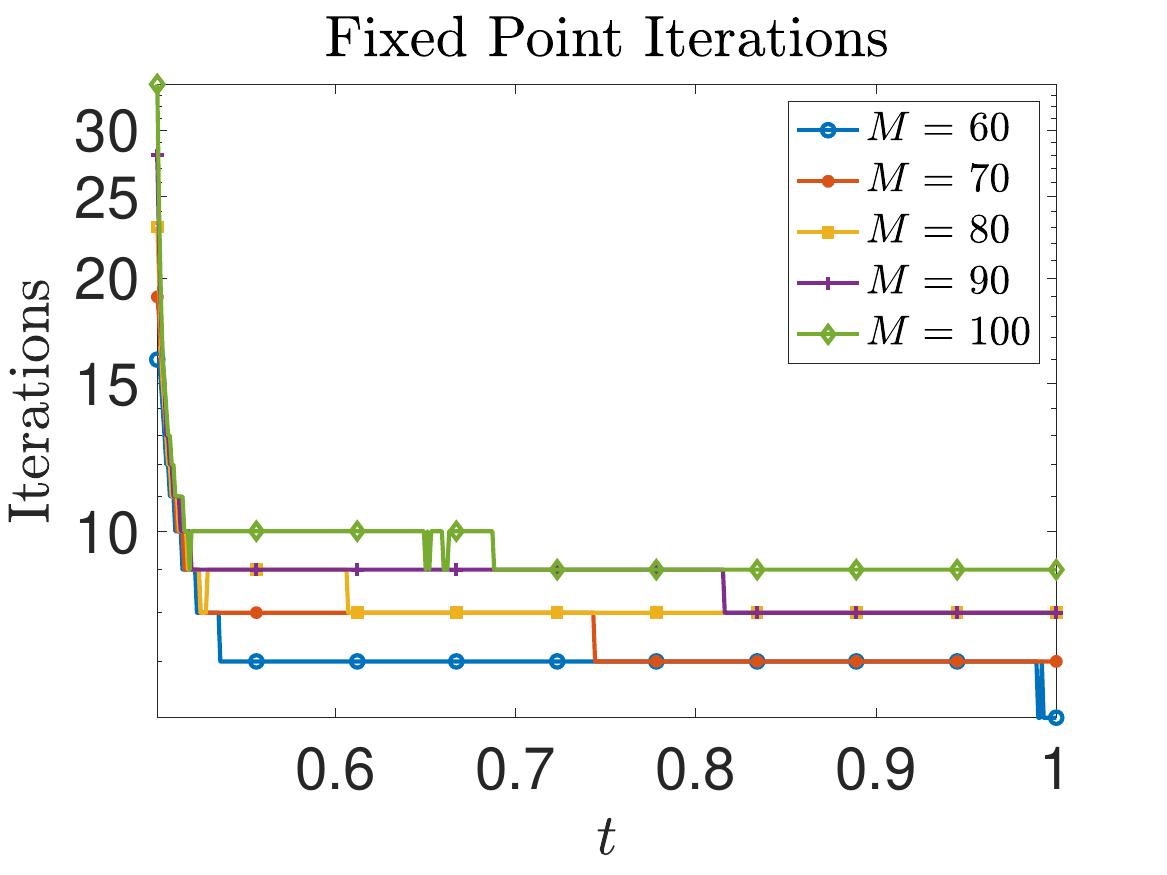}
    \captionof{figure}{Example \ref{example: non-local fokker-planck eq}: Number of fixed point iterations required to meet the stoping criterion along time.}%
  \label{fig:fokker planck non local num iter}
  \end{minipage} ~~
  \begin{minipage}[b]{.5\linewidth}
    \centering
    \begin{tabular}{|c|c|c|} \hline
  $M$ & mean number & max number\\ 
      & of iterations  & of iterations\\
      \hline
  $60$ & 7.18  & 16 \\
  \hline
  $70$ & 7.65 & 19 \\ 
  \hline
  $80$ & 8.34 & 23 \\
  \hline
  $90$ & 8.79 & 28 \\
  \hline
  $100$ & 9.52 & 34\\
  \hline
  \end{tabular}
    \captionof{table}{Example \ref{example: non-local fokker-planck eq}: The average number and maximum number of fixed point iterations required to meet the stopping criterion for different values of $M$.  The time step is $\Delta t = 0.01/10$ for all values of $M$.}%
  \label{table:non local fokker planck iter}
  \end{minipage}
\end{minipage}

\subsection{Landau equation with Maxwell kernel}
\label{example: Landau eq Maxwellian}
The remaining two examples correspond to the Landau equation \eqref{eq:Landau1} in 2D.  

We first consider the Maxwell collision kernel $a_{ij}(\bbx)=\frac{1}{16}(|\bbx|^2\delta_{ij}-x_ix_j)$.
It is in this case an analytical solution is known, and 
we refer to appendix A of \cite{CHWW20} for the derivation of the BKW solution, which is given by 
\begin{equation}
    \Phi(t,\bbx) = \frac{1}{2\pi R}\exp{\left(-\frac{|\bbx|^2}{2R}\right)\left(\frac{2R-1}{R} + \frac{1-R}{2R^2}|\bbx|^2 \right)}, \quad R = 1-\exp{(-t/8)/2}.
\end{equation}
We choose $t \in [0,5]$ with the initial condition $f(0,\bbx) = \Phi(0,\bbx)$. The computational domain is $[-4,4]^2$.  The particle solution is computed using a time step of $\Delta t = 0.01/8$ and $M = 40,45,50,55,60$. The left plot of Figure \ref{fig:BKW1} shows the evolution of the difference between the kinetic energy at time $t$ and its initial value. The right plot of Figure \ref{fig:BKW1} shows the decay of energy/entropy of the particle solution.  

\begin{figure}[!ht]
    \centering
    \includegraphics[width = .49\textwidth]{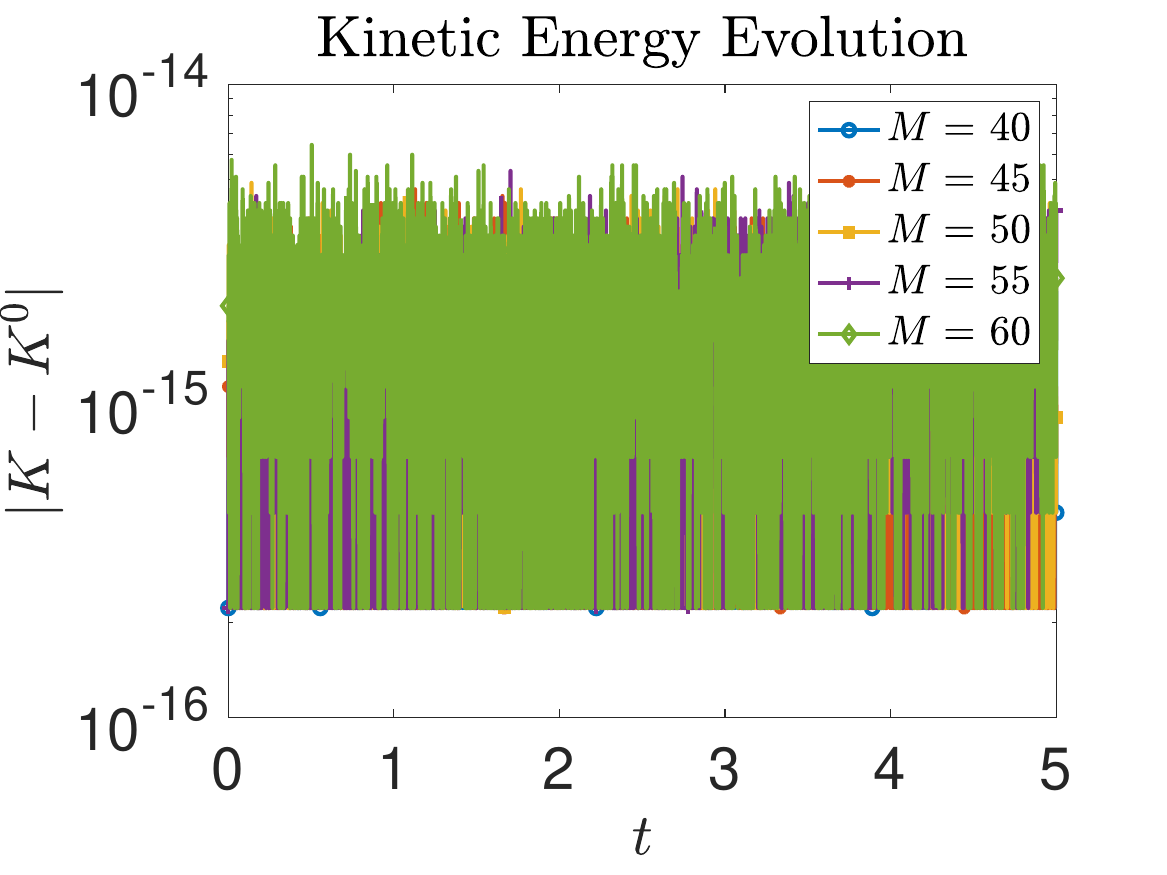}
    \includegraphics[width = .49\textwidth]{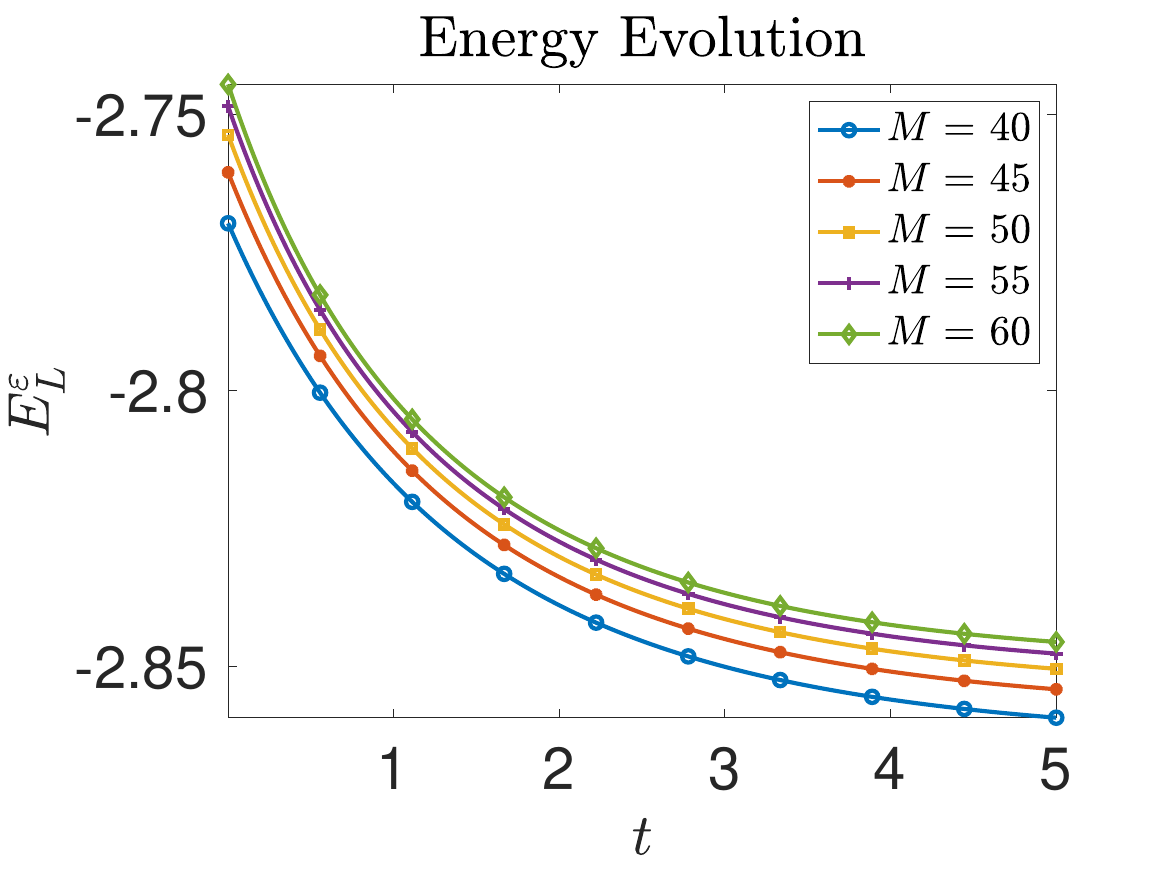}
    \caption{Example \ref{example: Landau eq Maxwellian}: Left: Time evolution of the difference between the kinetic energy at time $t$ and its initial value, for different values of $M$.  Right: Time evolution of the energy/entropy for different values of $M$.}
    \label{fig:BKW1}
\end{figure}

Figure \ref{fig:Landau Maxwellian num iter} shows the number of iterations the fixed point method needs to meet the stopping criterion along time.  Table \ref{table:Landau Maxwellian iter} shows the mean and maximum number of iterations for different values of $M$.   Figure \ref{fig:Landau Maxwellian ROC} is the $L^1,L^2,$ and $L^{\infty}$ errors vs the cell size $h$ at the final time $t = 5$ and shows that the particle method is approximately second order accurate. \\

\begin{minipage}{0.95\linewidth}
  \begin{minipage}[!b]{.45\linewidth}
    \centering
    \includegraphics[width=.99\textwidth]{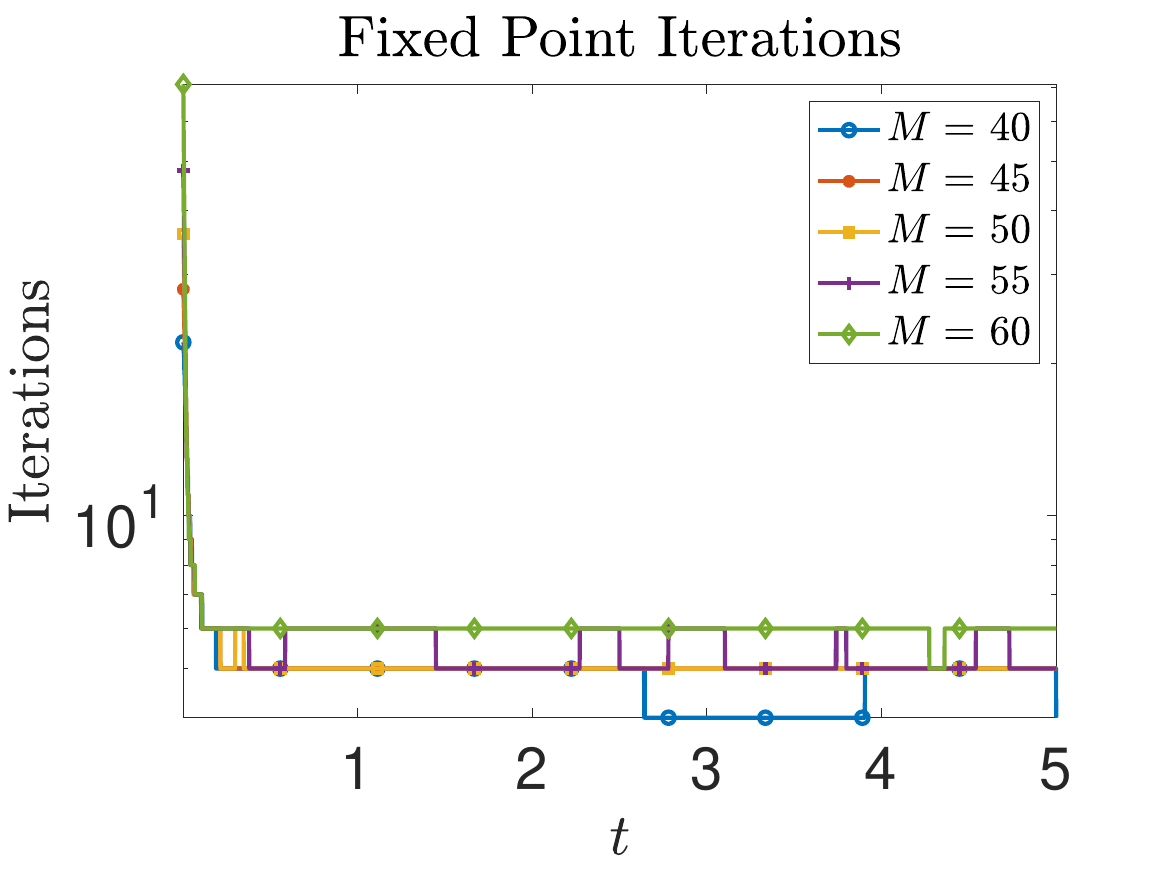}
    \captionof{figure}{Example \ref{example: Landau eq Maxwellian}: Number of fixed point iterations required to meet the stopping criterion along time.}%
  \label{fig:Landau Maxwellian num iter}
  \end{minipage} ~~
  \begin{minipage}[!b]{.5\linewidth}
    \centering
    \begin{tabular}{|c|c|c|}\hline
        $M$ & mean number & max number\\ 
          & of iterations  & of iterations\\
          \hline
      $40$ & 4.88  & 22 \\
      \hline
      $45$ & 5.15 & 28 \\ 
      \hline
      $50$ & 5.17 & 36 \\
      \hline
      $55$ & 5.54 & 48 \\
      \hline
      $60$ & 6.14 & 71\\
      \hline
    \end{tabular}
    \captionof{table}{Example \ref{example: Landau eq Maxwellian}: The average number and maximum number of fixed point iterations required to meet the stopping criterion for different values of $M$.  The time step is $\Delta t = 0.01/8$ for all values of $M$.}%
  \label{table:Landau Maxwellian iter}%
  \end{minipage}
\end{minipage}

\begin{figure}[!ht]
    \centering 
    \includegraphics[width = .49\textwidth]{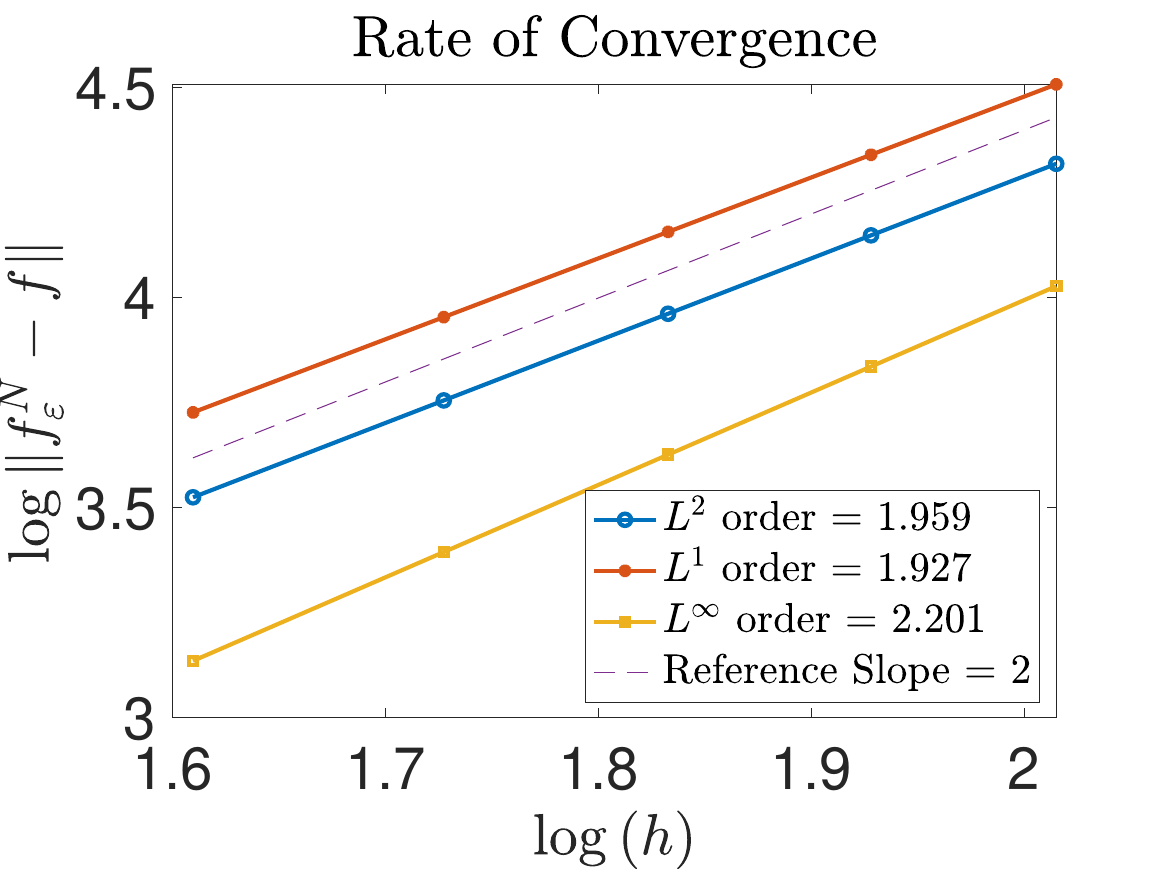}
    \caption{Example \ref{example: Landau eq Maxwellian}: $L^1, L^2$ and $L^{\infty}$ errors vs cell size $h$.}
    \label{fig:Landau Maxwellian ROC}
\end{figure}

\subsection{Landau equation with Coulomb kernel}
\label{example: Landau eq Coulomb}
The last example is the Landau equation with the Coulomb collision kernel $a_{ij}(\bbx)=\frac{1}{16}|\bbx|^{-3}(|\bbx|^2\delta_{ij}-x_ix_j)$.
We consider $t \in [0,20]$ with the initial condition
\begin{align*}
    &f(0,\bbx) = \frac{\pi}{4}\left(\exp{\left(-\frac{(\bbx - \bm{u}_1)^2}{2}\right)} + \exp{\left(-\frac{(\bbx-\bm{u}_2)^2}{2}\right)}\right),\\
    &\bm{u}_1 = (-2,1)^T, \quad \bm{u}_2 = (0,-1)^T.
\end{align*}
The computational domain is $ [-10,10]^2$.
The Coulomb case is the physically relevant case but there is no analytical solution to compare with. Therefore, we mainly demonstrate the conservation and dissipation properties of the proposed method. 

The particle solution is computed using a time step of $\Delta t = 0.1/2$ and $M = 40,45,50,55,60$. The plots in Figure  \ref{fig:Coulomb momentum} show the time evolution of the difference between the momentum at time $t$ and their initial values, in both the $x$ and $y$ directions.  The plot on the left of Figure \ref{fig:Coulomb Energy and Entropy} shows the difference between kinetic energy at time $t$ and its initial value, and the plot on the right shows the time evolution of the energy/entropy. As expected, our method is able to dissipate the energy and conserve the momentum and kinetic energy at the fully discrete level.

\begin{figure}[!ht]
    \centering
    \includegraphics[width = .49\textwidth]{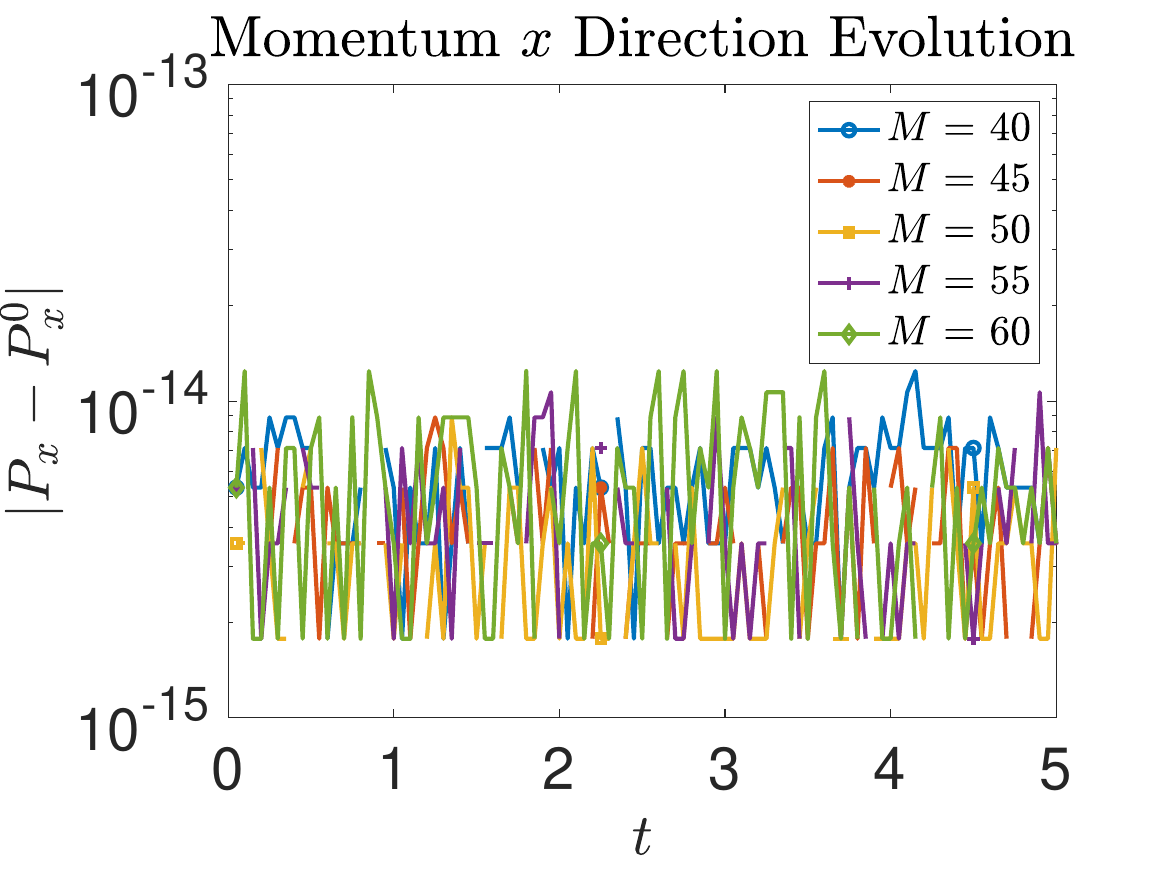}
    \includegraphics[width = .49\textwidth]{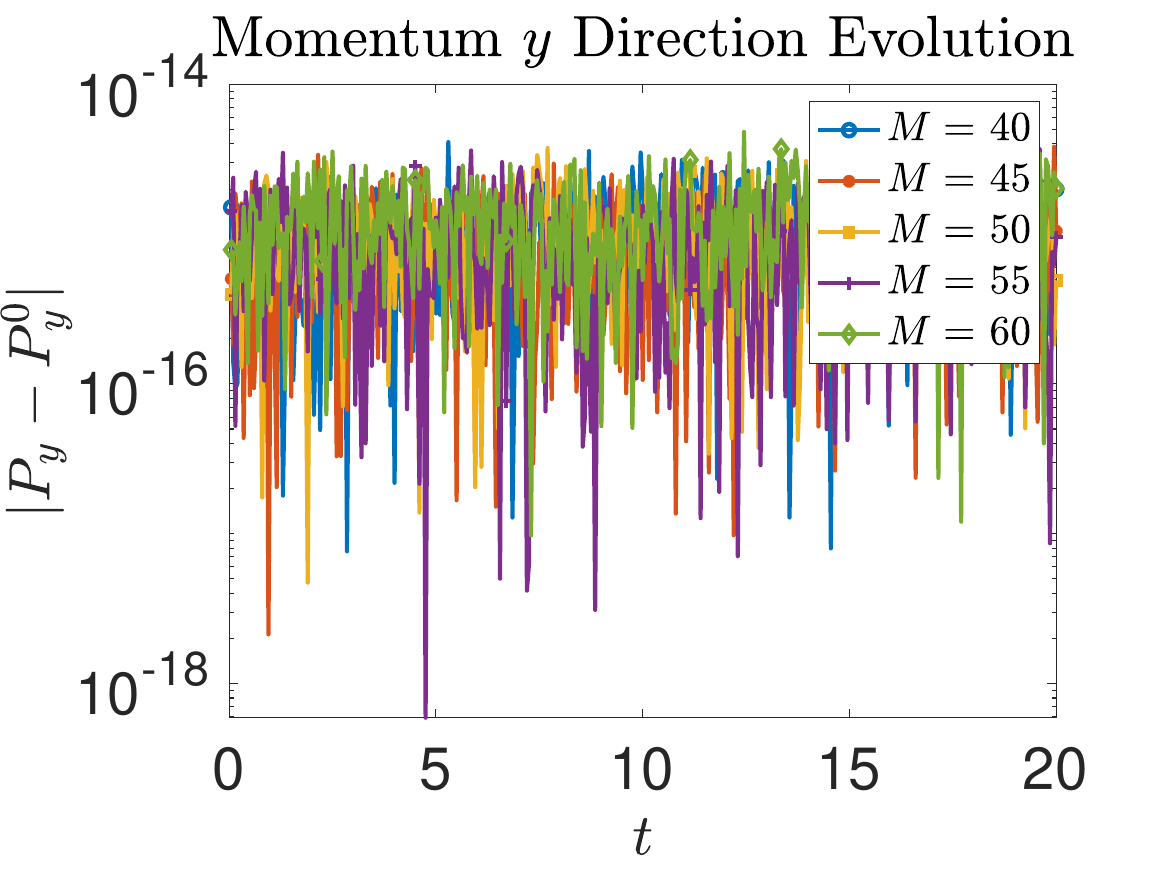}
    \caption{Example \ref{example: Landau eq Coulomb}: Time evolution of the difference between the momentum at time $t$ and their initial values.}
    \label{fig:Coulomb momentum}
\end{figure}

\begin{figure}[!ht]
    \centering
    \includegraphics[width = .49\textwidth]{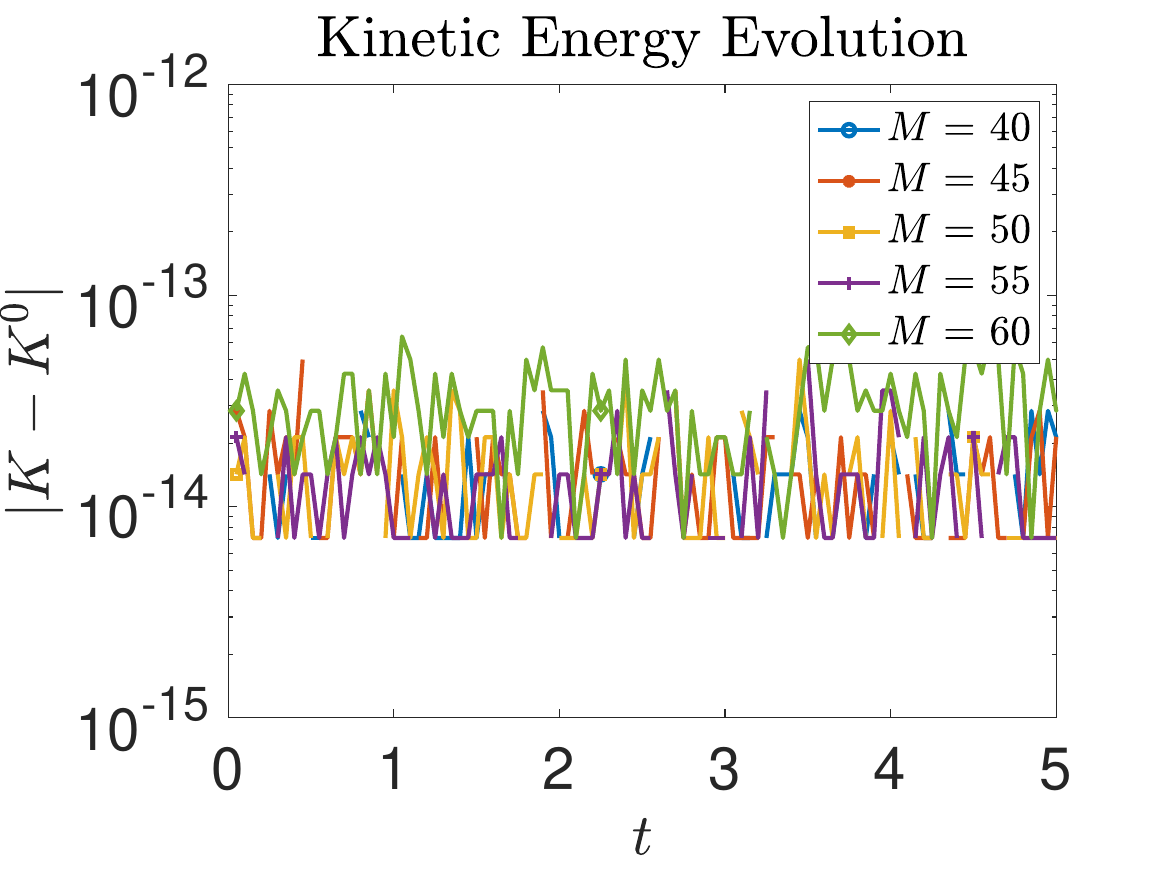}
    \includegraphics[width = .49\textwidth]{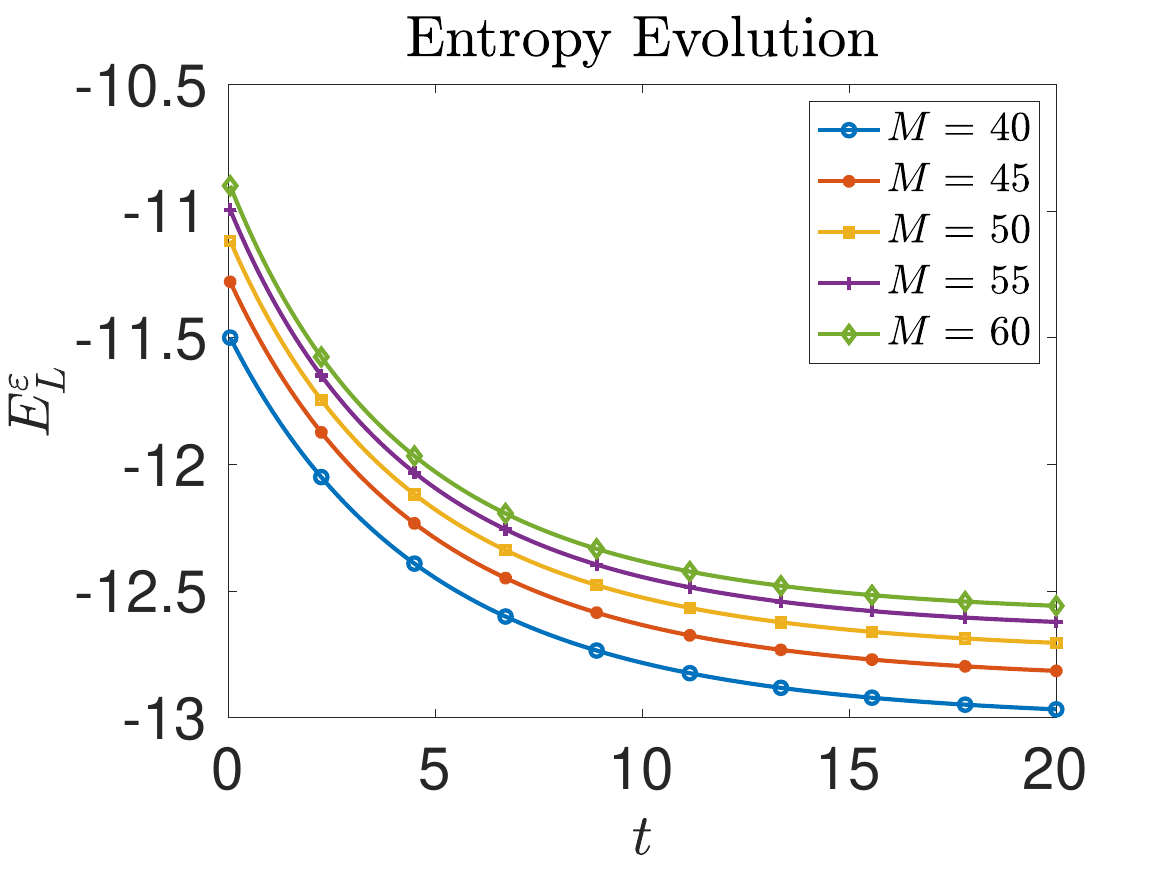}
    \caption{Example \ref{example: Landau eq Coulomb}: Left: Time evolution of the difference between the kinetic energy at time $t$ and its initial value. Right: Time evolution of the energy/entropy for different values of $M$.}
    \label{fig:Coulomb Energy and Entropy}
\end{figure}

 Figure \ref{fig:Landau Coloumb num iter} shows the number of iterations the fixed point method needs to meet the stopping criterion along time.  Table \ref{table:Landau Coloumb iter} shows the mean and maximum number of iterations for different values of $M$.

For the Landau equation with the Coulomb kernel, there is a great interest to search for additional dissipative quantities. For example, the recent breakthrough \cite{GS23} proved that the Fisher information is monotonically decreasing over time. Although our numerical method is not designed to capture this property, for exploration purpose we plot the time evolution of this quantity as well as the energy/entropy dissipation rate (which is conjectured to decay). The discrete Fisher information is given by 
\begin{equation*}
    F^{\varepsilon}(f^N) = \sum_{p = 1}^N\frac{1}{w_p}\left| \overline{\nabla_{\bbx_p} E_L^{\varepsilon}}(\bm X^{n+1},\bm X^n) \right|^2,
\end{equation*}
and the discrete energy/entropy dissipation rate is given by
\begin{equation*}
\begin{aligned}
D^{\varepsilon}(f^N) &= -\frac{1}{2}\sum_{p,q=1}^N w_p w_q \left(\frac{\overline{\nabla_{\bbx_p} E_L^{\varepsilon}}(\bm X^{n+1},\bm X^n)}{w_p} - \frac{\overline{\nabla_{\bbx_q} E_L^{\varepsilon}}(\bm X^{n+1},\bm X^n)}{w_q} \right) \\
&\hskip 18mm\times A(\overline{\bbx_p^n}-\overline{\bbx_q^n}) \left(\frac{\overline{\nabla_{\bbx_p} E_L^{\varepsilon}}(\bm X^{n+1},\bm X^n)}{w_p} - \frac{\overline{\nabla_{\bbx_q} E_L^{\varepsilon}}(\bm X^{n+1},\bm X^n)}{w_q} \right).
\end{aligned}
\end{equation*}
The left plot of Figure \ref{fig:Coulomb fisher info and dissipation} shows the time evolution of the Fisher information, and the right plot shows the time evolution of the energy/entropy dissipation rate. \\ 
 
\begin{minipage}{0.95\linewidth}
  \begin{minipage}[b]{.45\linewidth}
    \centering
    \includegraphics[width=.99\textwidth]{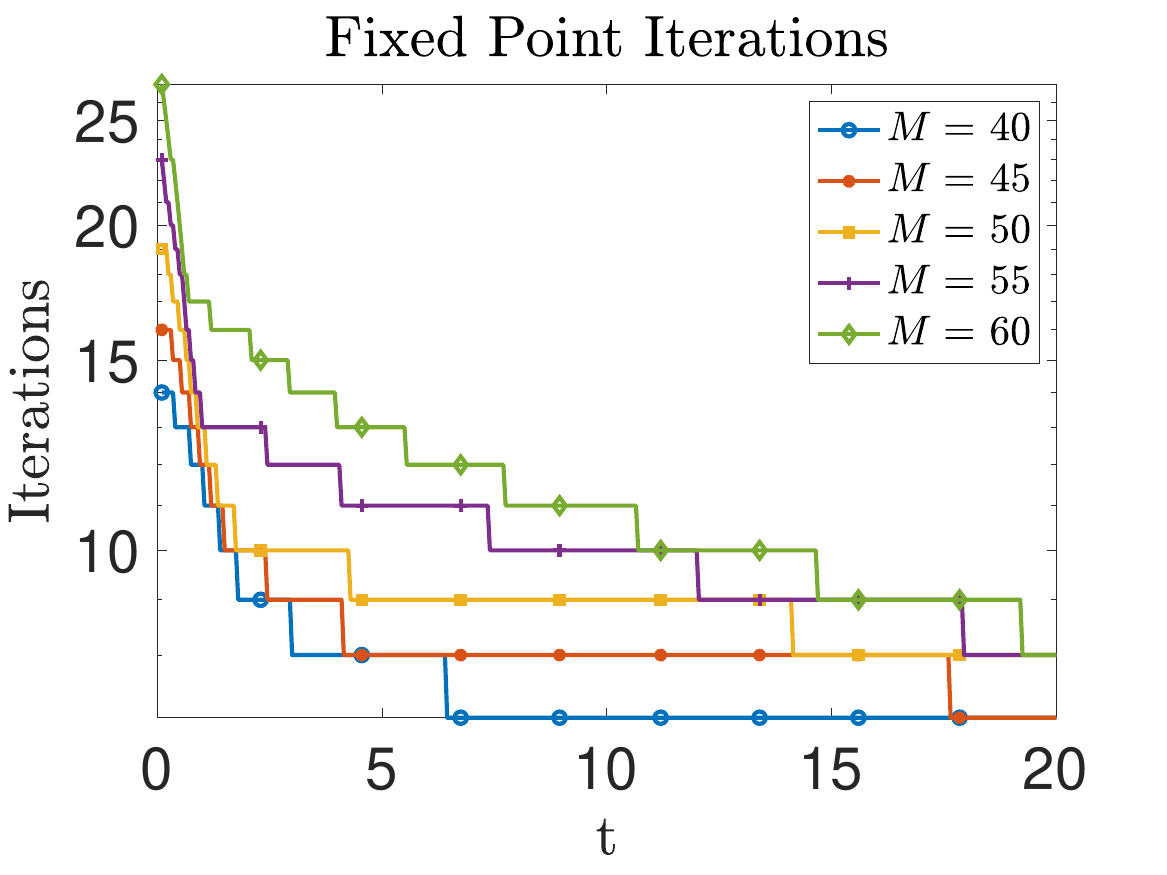}
    \captionof{figure}{Example \ref{example: Landau eq Coulomb}: Number of fixed point iterations required to meet the stopping criterion along time.
    \label{fig:Landau Coloumb num iter}}
  \end{minipage} ~~
  \begin{minipage}[b]{.5\linewidth}
    \centering
    \begin{tabular}{|c|c|c|}\hline
        $M$ & mean number & max number\\ 
          & of iterations  & of iterations\\
          \hline
      $40$ & 7.73  & 14 \\
      \hline
      $45$ & 8.46& 16 \\ 
      \hline
      $50$ & 9.28 & 19 \\
      \hline
      $55$ & 10.43 & 23 \\
      \hline
      $60$ & 11.57 & 28\\
      \hline
    \end{tabular}
    \captionof{table}{Example \ref{example: Landau eq Coulomb}: The average number and maximum number of fixed point iterations required to meet the stopping criterion for different values of $M$.  The time step is $\Delta t = 0.1/2$ for all values of $M$.}%
  \label{table:Landau Coloumb iter}
  \end{minipage}
\end{minipage}

\begin{figure}[!ht]
    \centering
    \includegraphics[width = .49\textwidth]{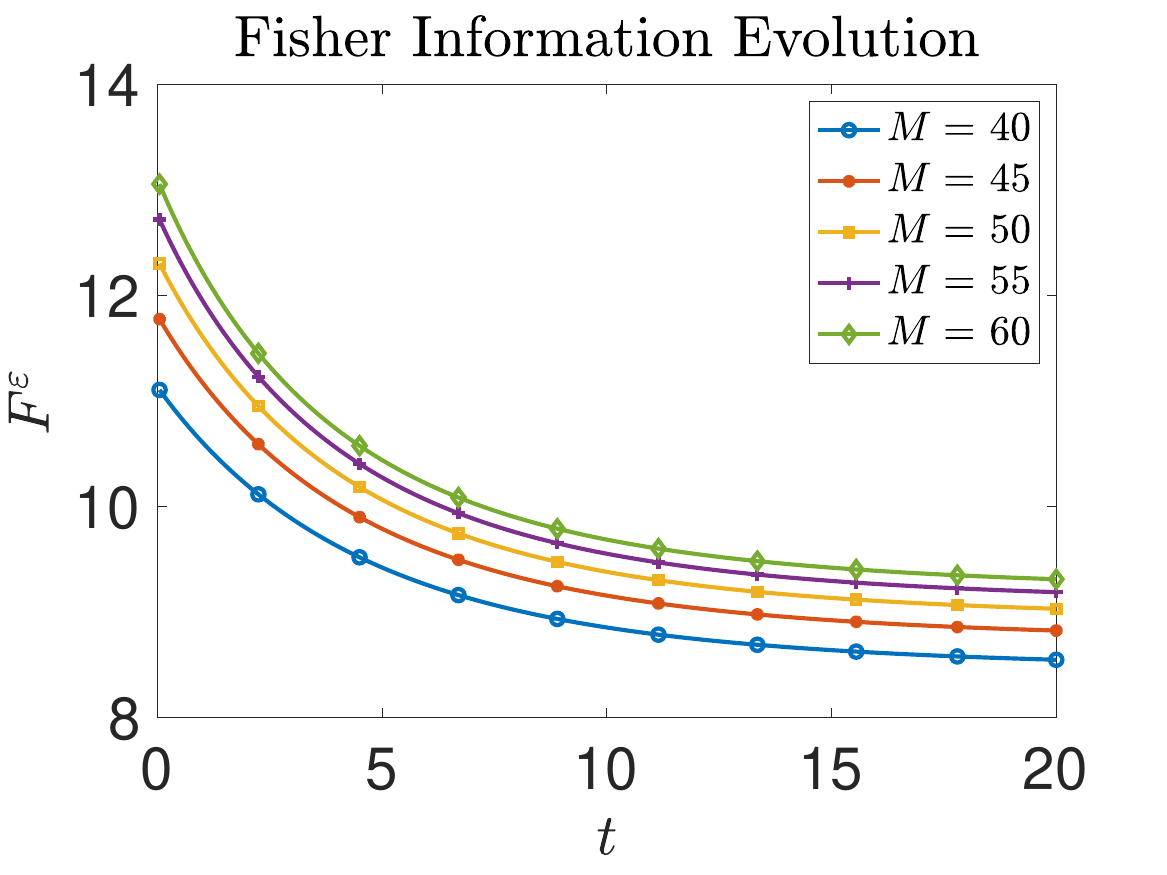}
    \includegraphics[width = .49\textwidth]{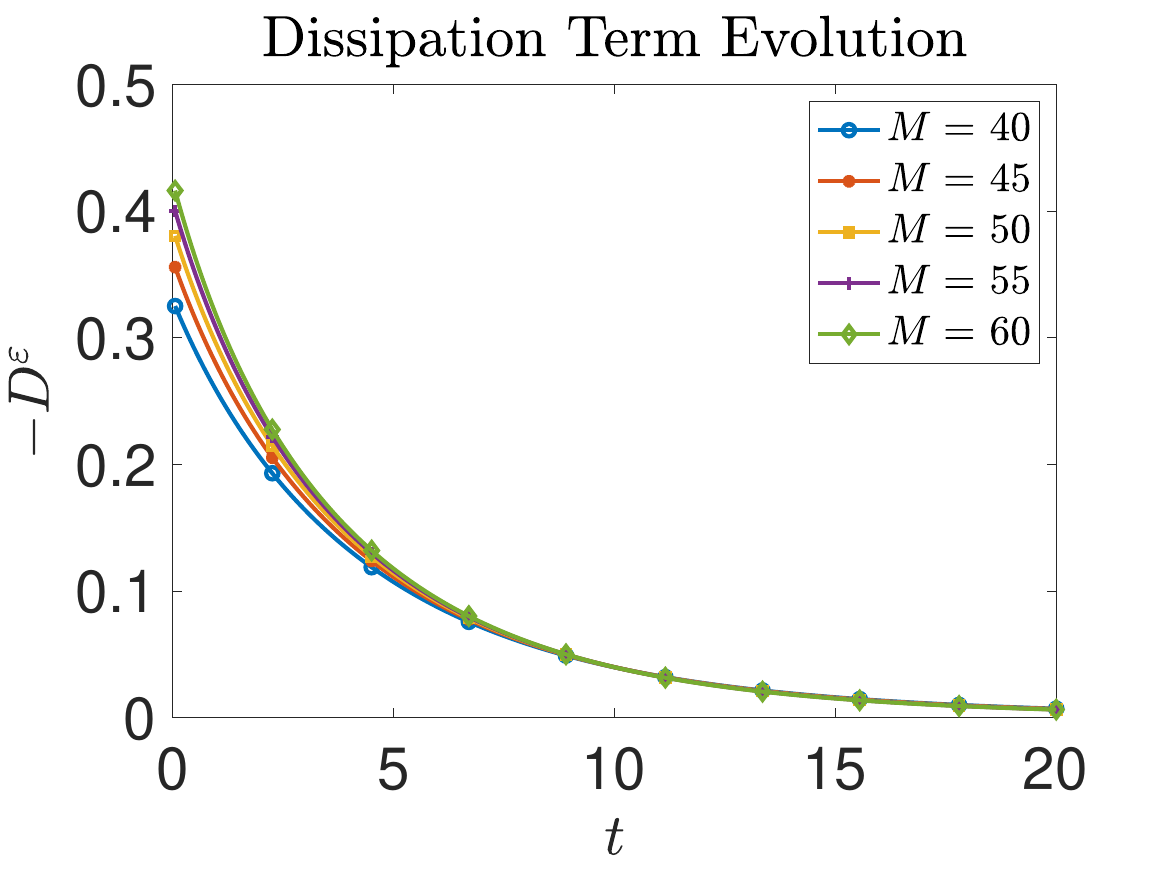}
    \caption{Example \ref{example: Landau eq Coulomb}: Left: Time evolution of the Fisher information. 
 Right: Time evolution of the energy/entropy dissipation rate.}
    \label{fig:Coulomb fisher info and dissipation}
\end{figure}

%\section{Conclusions}
%\label{sec:conclusions}

%Some conclusions here.

\section*{Acknowledgments}
ATSW thanks the Pacific Institute for the Mathematical Sciences for the support provided for JH's visit to the University of Northern British Columbia (UNBC) where this research was initiated when he was at UNBC.

\bibliographystyle{siamplain}
\bibliography{hu_bibtex}
\end{document}